\documentclass[11pt]{amsart}
\usepackage{a4wide}
\usepackage{amssymb}
\usepackage{amsthm}
\usepackage{amsmath}
\usepackage{amscd}
\usepackage{mathrsfs}
\usepackage{bbm}
\usepackage{color}
\usepackage[all]{xy}
\DeclareMathAlphabet{\mathpzc}{OT1}{pzc}{m}{it}
\usepackage{verbatim}

\numberwithin{equation}{section}

\theoremstyle{plain}
\newtheorem{theorem}{Theorem}[section]
\newtheorem{corollary}[theorem]{Corollary}
\newtheorem{lemma}[theorem]{Lemma}
\newtheorem{proposition}[theorem]{Proposition}
\theoremstyle{definition}
\newtheorem{definition}[theorem]{Definition}
\newtheorem{remark}[theorem]{Remark}

\theoremstyle{remark}

\newcommand{\A}{\mathbb{A}}
\newcommand{\R}{\mathbb{R}}
\newcommand{\Q}{\mathbb{Q}}
\newcommand{\Z}{\mathbb{Z}}
\newcommand{\N}{\mathbb{N}}
\newcommand{\C}{\mathbb{C}}

\renewcommand{\H}{\mathbb{H}}

\newcommand{\zxz}[4]{\begin{pmatrix} #1 & #2 \\ #3 & #4 \end{pmatrix}}

\newcommand{\leg}[2]{\left( \frac{#1}{#2} \right)}
\newcommand{\kzxz}[4]{\left(\begin{smallmatrix} #1 & #2 \\ #3 & #4\end{smallmatrix}\right) }

\newcommand{\im}{\operatorname{Im}}
\newcommand{\re}{\operatorname{Re}}

\newcommand{\odd}{\operatorname{oddity}}

\newcommand{\calA}{\mathcal{A}}

\newcommand{\calD}{\mathcal{D}}

\newcommand{\calF}{\mathcal{F}}
\newcommand{\calG}{\mathcal{G}}
\newcommand{\calH}{\mathcal{H}}

\newcommand{\calK}{\mathcal{K}}
\newcommand{\calL}{\mathcal{L}}
\newcommand{\calM}{\mathcal{M}}

\newcommand{\calQ}{\mathcal{Q}}

\newcommand{\calS}{\mathcal{S}}
\newcommand{\calT}{\mathcal{T}}
\newcommand{\calU}{\mathcal{U}}

\newcommand{\calZ}{\mathcal{Z}}

\newcommand{\frake}{\mathfrak e}

\newcommand{\bs}{\backslash}

\newcommand{\SL}{\operatorname{SL}}
\newcommand{\GL}{\operatorname{GL}}

\newcommand{\Hom}{\operatorname{Hom}}
\newcommand{\Aut}{\operatorname{Aut}}

\newcommand{\SO}{\operatorname{SO}}

\newcommand{\supp}{\operatorname{supp}}
\newcommand{\sig}{\operatorname{sig}}

\newcommand{\ord}{\operatorname{ord}}
\newcommand{\id}{\operatorname{id}}

\begin{document}

\title[Vector valued automorphic forms]{On vector valued automorphic forms for the Weil representation}
\author{Oliver Stein}
\address{Fakult\"at f\"ur Informatik und Mathematik\\ Ostbayerische Technische Hochschule Regensburg\\Galgenbergstrasse 32\\93053 Regensburg\\Germany}
\email{oliver.stein@oth-regensburg.de}

\begin{abstract}
  We develop a theory of vector valued automorphic forms associated to the Weil representation $\omega_f$ and  corresponding to vector valued modular forms transforming with the ``finite'' Weil representation $\rho_L$. For each prime $p$ we determine the structure of a vector valued spherical Hecke algebra depending on $\omega_f$, which acts on the space of automorphic forms.  
\end{abstract}

\maketitle

\section{Introduction}

Automorphic forms can be seen as a generalization of modular forms. They play a key role in defining (automorphic) $L$-functions and are thereby a vital part of the Langlands program. For all major types of modular forms a theory of automorphic forms has been developed, including the action of some suitable Hecke algebra. As a consequence, a standard $L$-function associated to common eigenform  in terms of its Satake parameters can be defined. The present paper is the first part of two articles which address the task of defining a standard $L$-function associated to a vector valued modular form transforming according to the Weil representation. This ``finite'' Weil representation is  defined in terms of an even integral lattice $L$ and the associated discriminant group $L'/L$, where $L'$ is the dual lattice of $L$.  In the setting of the present paper it is a representation of $\SL_2(\Z/N\Z)$ on the group ring $\C[L'/L]$, where $N$ means the level of $L$.

This kind of  modular forms plays an important role  in many recent papers. The weakly holomorphic forms of this type serve as input to a singular theta lift, which maps them to meromorphic modular forms on orthogonal groups whose zeroes and poles are supported on special divisors and which possess an infinite product expansion. This theta lift is the celebrated Borcherds lift (\cite{Bo}, \cite{Br1}), which has many applications in geometry, algebra and in the theory of Lie algebras. For instance, it is an interesting and widely studied problem to classify  reflective automorphic forms and thereby reflective lattices and Kac-Moody algebras (see e. g. \cite{Sch1} and \cite{Wa}).   
Most of the classical theory of modular forms has been established for these  modular forms over the past years (see e. g. \cite{Br1}, \cite{Br2}, \cite{St1} or \cite{Mue}). In \cite{BS}, the foundations of a Hecke theory were laid. In particular, a Hecke operator $T\kzxz{m^2}{0}{0}{1}$ for $m$ dividing $|L'/L|$  is defined by the action of the double coset $\SL_2(\Z)\kzxz{m^2}{0}{0}{1}\SL_2(\Z)$.

In his thesis \cite{We}, Werner introduced a generalized version of the Hecke operators defined in \cite{BS}. Most notably, these operators depend on a matrix $\kzxz{m}{0}{0}{1}$, where $m$ is coprime to $N$ and is not necessarily  a square (modulo $N$). Their construction is based on the extension of the ``finite'' Weil representation to the group $\GL_2(\Z/N\Z)$ on a larger representation space $X$, where $\C[L'/L]$ can be realized as an embedding into $X$. As a consequence, the classical space of vector valued modular forms (as usually considered in the literature and in the present paper) is a subspace of the considered vector valued modular forms for the extended Weil representation, which is not fixed by Werner's Hecke operators (see \cite{We}, p. 19, for details).  
Also, first steps towards a theory of vector valued automorphic forms corresponding to these more general vector valued modular forms were undertaken by Werner:  he defined automorphic forms for the (extended) Weil representation, however without  specifying the space these forms belong to. 
Additionally, for a  prime $p$  coprime to $N$ and not a square modulo $N$, Werner  introduced an adelic Hecke operator. His definition is  analogous to Gelbart's presentation of the scalar valued adelic Hecke operator (cf. \cite{Ge}, $\S$ 3, B, more specifically, Lemma 3.7) but rather ad-hoc. It may be interpreted as the convolution of a suitable vector valued characteristic function of the double coset $\GL_2(\Z_p)\kzxz{1}{0}{0}{p}\GL_2(\Z_p)$ and a vector valued automorphic form.  The convolution in this case is based on the concept of the Bochner integral. Finally, in \cite{We}, Theorem 53, it is proved that the adelic Hecke operator  corresponds to the Hecke operator $T\kzxz{p}{0}{0}{1}$ mentioned before.   Yet, to the best of my knowledge, so far a rigorous theory of vector valued automorphic forms does not exist. 
This motivates the main objectives of the present paper 
\begin{enumerate}
\item[i)]
  to develop a  theory of vector valued automorphic forms corresponding to vector valued modular forms for the finite Weil representation along the lines of \cite{Ge}, $\S$ 3 A, or \cite{KL}, Section 12.1-12.4.
\item[ii)]
  to  determine the structure of a fitting local vector valued spherical Hecke algebra and
  \item[iii)]
  to study its action on the space of automorphic forms by means of Hecke operators also defined by convolution and tailored to our setting. 
\end{enumerate}
These adelic Hecke operators correspond to Hecke operators $T\kzxz{p^{-k}}{0}{0}{p^{-l}}$ for any prime $p$  and $k+l\in 2\Z$. For $p$ dividing the level of $L$ these operators are a slight extension of the Hecke operators $T\kzxz{m^2}{0}{0}{1}$ for $m$ dividing the level of the lattice $L$ in \cite{BS}. This extension is specified in \cite{St2}. 

It is important to stress that our results do not hold in full generality, in particular not  for all lattices $L$:
\begin{enumerate}
\item[i)]
  We only consider  lattices $L$ of even signature, which implies that we restrict ourselves  to vector valued modular forms of integral weight on the group $\SL_2(\Z)$. Nevertheless, it seems reasonable that our theory could be extended to odd signature and modular forms of half integral weight on the metaplectic cover of $\SL_2(\Z)$.
\item[ii)]
  To determine the structure of the local vector valued Hecke algebras, we assume that the discriminant group $L'/L$ is {\it anisotropic}. It should be possible to remove this assumption. Without it, the structure of the involved Hecke algebras would be more complicated and the Satake map in Theorem \ref{thm:satake_isom} is more difficult to calculate. To obtain a smooth theory of adelic Hecke operators as in the present paper one would have to choose appropriate subalgebras.  
\item[iii)]
 The Hecke operators in \cite{BS} are parametrized by matrices whose determinant is either a square or a square modulo the level of  $L$. This is similar to the situation of the classical half-integral weight modular forms. As one of our goals is to establish the compatibility between the adelic Hecke operators in this paper and the ones on vector valued modular forms, we restrict ourselves to operators $T\kzxz{p^{-k}}{0}{0}{p^{-l}}$ with $k+l\in 2\Z$. It might be possible to extend our results to Hecke operators where $p$ is a square modulo the level of $L$ and $k+l$ is not even. In this case the definition of the  operators $T\kzxz{p^{-k}}{0}{0}{p^{-l}}$ are covered by \cite{BS}, the definition of the generators of the involved Hecke algebras and the adelic Hecke operators should be the same as in the present paper and Theorem \ref{thm;action_hecke_op_notcoprime} is likely to carry over to this case as well. 
  \end{enumerate}

The second article in this series of papers then deals with the definition of a standard $L$-function for these vector valued modular forms and studies its analytic properties.

The paper at hand should be seen as the starting point for a more comprehensive study on vector valued automorphic forms. For one, vector valued modular forms for the Weil representation enjoy relations to scalar valued elliptic modular forms for $\Gamma_0(N)$ (cf. \cite{Sch}), $N$ the level of the lattice $L$ (see Section \ref{sec:preliminaries} for details), and to Jacobi forms of lattice index (see e. g. \cite{Wa}). For both of these types of modular forms exists a well established theory of automorphic forms (see e. g. \cite{Ge} and \cite{Mu1}). It would be interesting to study how exactly the vector valued automorphic forms are related to the elliptic automorphic forms and the Jacobi automorphic forms.\newline
Automorphic forms are closely connected to automorphic representations (cf. \cite{Ge}, Section 5). It should be worthwhile to investigate whether a similar theory can be build in the case of vector valued modular forms, and if so, what relations there are.

Let us describe the content of the paper in more detail. To this end, let $(L,(\cdot,\cdot))$ be an even lattice of even rank $m$ and type $(b^+, b^-)$ with (even) signature $\sig(L)=b^+-b^-$ and level $N$. Associated to the bilinear form $(\cdot,\cdot)$ there is a quadratic form $q$. The modulo 1 reduction of $(\cdot, \cdot)$ and $q$  defines a bilinear form and a quadratic form, respectively, on the discriminant group $D=L'/L$.  The Weil representation $\rho_L$ is a representation of $\Gamma = \SL_2(\Z)$ on the group ring $\C[D]$. As usual, we denote with  $\Z_p$ the ring of $p$-adic integers. As will be explained later in the paper, $\rho_L$ is isomorphic to a finite dimensional subrepresentation of the  Weil representation $\omega_f=\bigotimes_{p<\infty}\omega_p$ (originally defined by Weil \cite{Wei}) of $\SL_2(\widehat{\Z})$  on a space $S_L$ (isomorphic to $\C[D]$). Here $\widehat{\Z}$ is defined by  $\prod_{p<\infty}\Z_p$.
We have the relation
\[
\rho_L(\gamma) = \omega_f(\gamma_f)
\]
for all $\gamma\in \Gamma$, where $\gamma_f$ is the projection of $\gamma$ into $\SL_2(\widehat{\Z})$ (see e. g. \cite{YY}, p. 3456). Based on this relation and the extension process of $\rho_L$ to a subgroup of $\GL_2(\Q)$ in \cite{BS}, Section 3, we transfer this process to the extension of $\omega_f$ to some subgroup of $\GL_2(\A_f)$, where $\A_f$ means the finite adeles. For the details see Chapter \ref{subsec:weil_repr_adelic}. 

For $\kappa\in \Z$, a vector valued modular form of weight $\kappa$ and type $\rho_L$ is a holomorphic function $f:\H\rightarrow \C[D]$, which  satisfies
\[
f(\gamma\tau) = (c\tau +d)^\kappa \rho_L(\gamma)f(\tau) 
\]
for all $\gamma =\kzxz{a}{b}{c}{d}\in \Gamma$ and all $\tau$ in the complex upper half plane $\H$, and is holomorphic at the cusp $\infty$. We denote the space of all such functions with $M_\kappa(\rho_L)$ and write $S_\kappa(\rho_L)$ for the subspace of cusp forms. 
Now let $\A$ be the ring of adeles,  $\calG(\Q)$ a subgroup of $\GL_2(\Q)^+$ and \[
\calG(\A) =\sideset{}{'} \prod_{p\le \infty}\calQ_p = \left\{(g_p)\in \prod_{p\le \infty}\calQ_p\; |\; g_p\in \calK_p \text{ for almost all primes } p\right\},
\]
where $\calQ_p$ and $\calK_p$ is a subgroup of $GL_2(\Q_p)$ and $\GL_2(\Z_p)$, respectively. We assign to $f$ above a function $F_f:\calG(\Q)\setminus \calG(\A)\rightarrow S_L$ by means of strong approximation for the group $\calG(\A)$. For $g = g_\Q(g_\infty \times k)$ we put
\[
F_f(g) = \omega_f(k)^{-1}j(g_\infty,i)^{-\kappa}f(g_\infty i).
\]
Here $g_\Q\in \calG(\Q)$, $g_\infty\in \calQ_\infty \subset \GL_2^+(\R)$ and $k\in \calK = \prod_{p<\infty}\calK_p$ (see \eqref{def:p_adic_subgroups} for the definitions of the groups above). In Proposition \ref{rem:correspondence_L2} and Lemma \ref{lem:cuspidal} we will show that $F$ is a cuspidal vector valued automorphic form of type $\omega_f$, which can be seen as a vector valued analogue of a scalar valued cuspidal automorphic form. Moreover, Theorem \ref{thm:correspondence_mod-adelic} states that space of these functions, satisfying further properties, is isomorphic to $S_{\kappa}(\rho_L)$. We denote this space with $A_\kappa(\omega_f)$. The inverse map can also be  given explicitly: For $F\in A_\kappa(\omega_f)$ it can be proven that $f_F$, specified by
\[
\tau \mapsto f_F(\tau)=j(g_\tau,i)^\kappa F(g_\tau\times 1_f)
\]
with $g_\tau=\kzxz{1}{x}{0}{1}\kzxz{y^{1/2}}{0}{0}{y^{-1/2}}$ and $\tau = g_\tau i = x+iy\in\H$, is indeed an element of $S_\kappa(\rho_L)$. 
Note that the definition of $F_f$ has already occurred in the work of Werner (\cite{We}). The function $f_F$ can be found in Kudla's paper \cite{Ku}. However, as far as I know, a  theory of vector valued automorphic forms (as outlined above) has not yet appeared in the literature.

It is well known that for each prime $p$ the spherical Hecke algebra $\calH(\GL_2(\Q_p)//\GL_2(\Z_p))$ of locally constant, compactly supported and complex-valued functions, which additionally satisfy
\[
f(k_1 g k_2) = f(g)
\]
for all $k_1,k_2\in \GL_2(\Z_p)$ and all $g \in \GL_2(\Q_p)$, acts on the space of scalar valued automorphic forms. This action is compatible with the action of Hecke operators on the space of cusp forms (see for example \cite{BP}, \cite{Ge} or \cite{KL}). 
In this paper we define for each prime $p$ a spherical Hecke algebra $\calH(\calQ_p//\calK_p,\omega_p)$ of type $\omega_p$ as follows: Let  $L_p = L\otimes \Z_p$ and $S_{L_p}$ the representation space as above, but associated to the  $p$-adic lattice $L_p$. Note that $S_L=\bigotimes_{p< \infty}S_{L_p}$ (see Chapter \ref{subsec:weil_repr_adelic} for  details). Then $\calH(\calQ_p//\calK_p,\omega_p)$ consists of all locally constant and compactly supported functions $f:\calQ_p\rightarrow S_{L_p}$, which satisfy
\[
f(k_1 g k_2) = \omega_p(k_1)\circ f(g)\circ \omega_p(k_2)
\]
for all $k_1,k_2\in \calK_p$ and all $g\in \calQ_p$. Hecke algebras of this type are well known and studied in the literature (\cite{BK}, \cite{Ho}, \cite{He}).
A ``tool'' to investigate the structure of Hecke algebras (for a pair of groups $(G,K)$ with suitable properties), vector valued or not, is the Satake map (see \cite{Sa} or \cite{Ca}), whose image is a Hecke algebra easier to understand.
Under the assumption that $L_p'/L_p$ is {\it anisotropic}, we determine a set of generators of $\calH(\calQ_p//\calK_p,\omega_p)$ and connect it to a scalar valued Hecke algebra by means of a variant of the classical Satake map (see \eqref{def:satake_map}), which we adopt from \cite{He} and which is suitable in our situation. Both cases, $p | |D|$ and $p\nmid |D|$, are treated. In the first instance, the structure of $\calH(\calQ_p//\calK_p,\omega_p)$ is way more complicated than in the second one as the Weil representation $\omega_p$ is non-trivial. We obtain for a subalgebra $\calH^+(\calQ_p//\calK_p,\omega_p)$ the following theorem (see Thm. \ref{thm:satake_isom} and  Section \ref{sec:preliminaries} and Subsection \ref{subsec:Hecke_algebra_ncoprime} for the details): 

\begin{theorem}
Let $p$ be a prime dividing $|D|$ and the $p$-group $D_p$ of $D$ anisotropic. Further, let $\calM_p$, $\calD_p$, and $N(\Z_p)$ be as in \eqref{def:p_adic_subgroups} and  $W$ the usual Weyl group. 
Then the Hecke algebras $\calH^+(\calQ_p//\calK_p, \omega_p)$ and $\calH(\calM_p//\calD_p, \omega_{p_{|S_{L_p}^{N(\Z_p)}}})^W$ are isomorphic, where
$S_{L_p}^{N(\Z_p)}$ is subspace of $S_{L_p}$, which is invariant under the action of $N(\Z_p)$ via the Weil representation.
  \end{theorem}

If $p$ is coprime to $|D|$, the Weil representation $\omega_p$ is trivial and we essentially recover the classical result for $\GL_2$ (see Theorem \ref{thm:satake_isom_coprime}).

Subsequently, we define an action of $\calH(\calQ_p//\calK_p,\omega_p)$ on $A_\kappa(\omega_f)$, which can be interpreted as vector valued analogue of the versions in \cite{BP} or \cite{Mu1}. Under the assumption that the lattice $L$ is unimodular, our adelic Hecke operators are closely related to the classical scalar valued adelic Hecke operators (see Remark \ref{rem:prop_adelic_hecke_op}, iii), for details). However, Werner's adelic Hecke operator and the ones in this paper can not be compared without further work since
\begin{itemize}
\item[i)]
  the Hecke operator in this paper is defined in terms of all primes $p$ and not only for those with $p\nmid |L'/L|$,
\item[ii)]
  the Hecke operator in this paper is defined by the action of a local Hecke algebra while Werner's definition adopts Gelbart's construction without  involving the action of a Hecke algebra,
\item[iii)]
  the Hecke operator in this paper is only defined in terms of matrices of $\GL_2(\Q_p)$ whose determinant is a square and not for the remaining matrices of $\GL_2(\Q_p)$. 
\end{itemize}

As is common in the literature on automorphic forms, we show in Theorem \ref{thm;action_hecke_op_notcoprime} that this action is compatible with the action of Hecke operators on $S_\kappa(\rho_L)$. More precisely:

\begin{theorem}
  Let $p$ be a prime and $(k,l)\in \Lambda_+$.
  If $p$ divides $|D|$, let $T_{k, l}\in \calH^+(\calQ_p//\calK_p,\omega_p)$ as in Corollary \ref{cor:struct_hecke_algebra}. If $(p,|D|) = 1$, let $T_{k, l}=\mathbbm{1}_{\calK_pm(p^k, p^l)\calK_p}\id_{S_{L_p}}\in \calH(\calQ_p//\calK_p,\omega_p)$ be as in  Theorem \ref{thm:satake_isom_coprime} ($\mathbbm{1}_{\calK_pm(p^k, p^l)\calK_p}$ being the characteristic function of the double coset $\calK_pm(p^k, p^l)\calK_p$). Further, let  $\calT^{T_{k,l}}$  be  the adelic Hecke operator as in Definition \ref{deinition:adelic_hecke_op} and \newline$T(m(p^{-k},p^{-l}))$ the Hecke operator as defined in Section \ref{sec:preliminaries}. 
  Then for any $f\in S_\kappa(\rho_L)$ we have 
  \begin{equation}
    \calT^{T_{k,l}}(F_f) = F_{p^{(k+l)(\frac{\kappa}{2}-1)}T(m(p^{-k},p^{-l}))f},
  \end{equation}
  where $F_f$ is the automorphic form related to $f$ via the adelization map $\mathscr{A}$ as mentioned above. 
\end{theorem}
This result resembles the corresponding classical statement as e. g. in \cite{Ge}, Lemma 3.7 or Prop. 13.6 in \cite{KL}.  


  

\section{Preliminaries on the ``finite'' Weil representation, vector valued modular forms and some notation}\label{sec:preliminaries} 
In this section we provide some notation used throughout the paper and briefly summarize some facts on lattices, discriminant forms and the ``finite'' Weil representation. We also recall the definition of vector valued modular forms for the Weil representation and some related theory relevant for the present paper. 

As usual, we let $e(z)$, $z\in \C$, be the abbreviation for $e^{2\pi i z}$. For any prime $p\in \Z$ by $\Q_p$ we mean the field of $p$-adic numbers and by $\Z_p$ its ring of $p$-adic integers; $|\cdot|_p$ is the $p$-adic absolute value and $\ord_p(\cdot)$ the $p$-adic valuation of $\Q_p$. We write $\A$ for the adele ring and $\A^\times$ for the idele group. By $\A_f$ we mean the set of finite adeles.

The following matrix groups  appear frequently in the paper.
\begin{equation}\label{def:p_adic_subgroups}
    \begin{split}
      &\calG(R) = \{M\in \GL_2(R)\; |\; \det(M) = r^2\} \text{ for any commutative ring $R$ with $1$},\\
      & N(\Q_p) = \{ (\kzxz{1}{r}{0}{1}, 1)\; |\; r\in \Q_p\} \text{ and } N(\Z_p) \text{ accordingly},\\
      &\calQ_p=\calG(\Q_p),\\
      & \calK_p=\calG(\Z_p),\\
      &\calM_p=\{(M, r) = (\kzxz{r_1}{0}{0}{r_2}, r)\in \GL_2(\Q_p)\times \Q_p^\times \; |\; \det(M)=r^2\} \text{ and}, \\
      &\calD_p=\calM_p\cap \calK_p.
    \end{split}
\end{equation}
To ensure a more readable exposition, we use abbreviations for certain elements of these groups:
\begin{align*}
  n\_(c) = \kzxz{1}{0}{c}{1}, \quad n(b) = \kzxz{1}{b}{0}{1}, \quad m(s) = \kzxz{s}{0}{0}{s^{-1}}, \quad m(t_1,t_2) = \kzxz{t_1}{0}{0}{t_2}\text{ and } w=\kzxz{0}{1}{-1}{0}. 
\end{align*}

The symbol $\iota_p$ occurs in some places of the paper:  For $x_p\in \calQ_p$ let  $\iota_p(x_p) = (\alpha_q)_{q\le\infty}\in \calG(\A)$ with $\alpha_q=1_q$ for $q\not=p$ and $\alpha_p=x_p$, where $1_q$ means the unit matrix in $\calQ_q$. 
Moreover, we will make frequently use of the following subsets of $\Z^2$:
\begin{align*}
    &\Lambda = \{(k,l)\in \Z^2\; |\; k \le l \text{ and } k+l\in 2\Z\} \text{ and }\\
    & \Lambda_+ = \{(k,l)\in \Lambda\; |\;  k,l\ge 0\}. 
    \end{align*}
  Finally, as usual, we write $\H=\{\tau \in \C\; |\; \im(\tau) > 0\}$ for the complex upper half plane, let  $\leg{\cdot}{d}$ be  the Legendre symbol and use $\mathbbm{1}_A$ as a symbol for the characteristic function of the set $A$. 
  
Let $L$ be a lattice of rank $m$  equipped with a symmetric $\Z$-valued bilinear form $(\cdot,\cdot)$ such that the associated quadratic form
\[
q(x):=\frac{1}{2}(x,x),\quad x\in L,
\]
takes values in $\Z$. We assume that $m$ is even, $L$ is non-degenerate and denote its type by $(b^+,b^-)$ and its signature $b^+-b^-$ by $\sig(L)$. Note that $\sig(L)$ is also even. We stick with these assumptions on $L$ for the rest of this paper unless we state it otherwise. Further, let 
\[
L':=\{x\in V=L\otimes \Q\; |\; (x,y)\in\Z\quad \text{ for all } \; y\in L\}
\]
be the dual lattice of $L$. 
Since $L\subset L'$, the elementary divisor theorem implies that $L'/L$ is a finite group. We denote  this group by $D$. The modulo 1 reduction of both, the bilinear form $(\cdot, \cdot)$ and the associated quadratic form, defines a $\Q/\Z$-valued bilinear form $(\cdot,\cdot)$ with corresponding $\Q/\Z$-valued quadratic form on $D$. We call $D$ combined with $(\cdot,\cdot)$ a discriminant form, discriminant group or a quadratic module. We call it anisotropic, if $q(\mu) = 0$ holds only for $\mu=0$. 

It is well known that any discriminant form of odd order can be decomposed into a direct sum of quadratic modules of the form
\[
\calA_{p^k}^t = \left(\Z/p^k\Z, \;\frac{tx^2}{p^k}\right).
\]
 An anisotropic quadratic module of odd order consists of  $p$-groups $D_p$, which can  either be written as $\calA_p^t$ or as a direct sum $\calA_p^t\oplus \calA_p^1$. For further details we refer to  \cite{BEF}. 

The ``finite'' Weil representation $\rho_L$ is a representation of $\Gamma = \SL_2(\Z)$ on the group ring $\C[D]$. We denote the standard basis of $\C[D]$  by $\{\frake_{\lambda}\}_{\lambda\in D}$.
On the standard generators  
\begin{equation}\label{eq:symplectic_group_gen}
S=\zxz{0}{-1}{1}{0},\quad T= \zxz{1}{1}{0}{1}, 
\end{equation}
of $\Gamma$ $\rho_L$ is given by
\begin{equation}\label{eq:weil_finite_repr}
\begin{split}
&\rho_L(T)\frake_\lambda := e(q(\lambda))\frake_\lambda,\\
& \rho_L(S)\frake_\lambda:=\frac{e(-\sig(L)/8)}{|D|^{1/2}}\sum_{\mu\in D}e(-(\mu,\lambda))\frake_\mu.
\end{split}
\end{equation}
We denote by $N$ the level of the lattice $L$. It is the smallest positive integer such that $Nq(\lambda)\in \Z$ for all $\lambda\in L'$.  One can prove that the Weil representation $\rho_{L}$ is trivial on $\Gamma(N)$, the principal congruence subgroup of level $N$. This can for example be deduced from \cite{Sch}, Prop. 4.2. Therefore, $\rho_{L}$ factors over the finite group
  \[
  \Gamma/\Gamma(N)\cong \SL_2(\Z/N\Z).
  \]
  For the rest of this paper we suppose that $N$ is {\it odd}.

  We now define vector valued modular forms of type $\rho_L$. With respect
to the standard basis of $\C[D]$ a function $f: \H\rightarrow \C[D]$ can be written in the form
\[
f(\tau) = \sum_{\lambda\in D}f_\lambda(\tau)\frake_\lambda.
\]
The following operator generalises the usual Petersson slash operator to the space of all those functions. For $\kappa\in \Z$ we define
\begin{equation}\label{eq:slash_operator}
(f\mid_{\kappa,L}\gamma)(\tau) =
j(\gamma,\tau)^{-\kappa}\rho_L(\gamma)^{-1}f(\gamma\tau),
\end{equation}
where
\[
j(\gamma,\tau) = \det(\gamma)^{-1/2}(c\tau+d)
\]
is the usual automorphy factor if $\gamma=\kzxz{a}{b}{c}{d}\in \GL_2^+(\R)$.

A holomorphic function $f: \H\rightarrow \C[D]$ is called a vector valued  modular
form of weight $ \kappa $ and type $ \rho_L $ for $ \Gamma $ if
$ f\mid_{\kappa,L}\gamma= f $ for all $\gamma\in \Gamma$,
and if $f$ is holomorphic at the cusp $\infty$. Here the last condition means that all Fourier coefficients $c(\lambda, n)$ of $f$ with $n <  0$ vanish. If in addition $c(\lambda,n)$ with $n = 0$ vanish, we call the corresponding modular form a cusp form. We denote by $M_\kappa(\rho_L)$ the
space of all such modular forms, by $S_\kappa(\rho_L)$ the subspace of cusp forms. For more details see e.g. \cite{Br1} or \cite{BS}.

The Petersson scalar product on $S_\kappa(\rho_L)$ is given by
\begin{equation}\label{eq:petersson_scalar}
(f,g) =
\int_{\Gamma\backslash \H}\langle f(\tau), g(\tau)\rangle \im \tau^\kappa d\mu(\tau)
\end{equation}
where $\tau = x+iy$ and 
\[
d\mu(\tau) = \frac{dx\;dy}{y^2}
\]
denotes the hyperbolic volume element and 
  \begin{equation}\label{eq:scalar_product_group_ring}
    \left\langle \sum_{\lambda\in D}a_\lambda\frake_\lambda,\sum_{\lambda\in D}b_\lambda\frake_\lambda \right\rangle =\sum_{\lambda\in D}a_\lambda \overline{b_\lambda}.
  \end{equation}
 is the standard scalar product on $\C[D]$.

  Let $d$ an integer. By $g_d(D)$ we denote the Gauss sum
  \begin{equation}\label{eq:gauss_sum_d}
    g_d(D)=\sum_{\lambda\in D}e(dq(\lambda))
  \end{equation}
  and we set $g(D) = g_1(D)$. 
Since fractions of these Gauss sums are of some relevance in this paper, we gather some facts on the sums $g_d(D)$ and quotients thereof.
\begin{lemma}
  \begin{enumerate}
    \item[i)]
  The Gauss sums $g_d(D)$ satisfy the properties
  \begin{align*}
   & g_{-d}(D) = \overline{g_d(D)} \\
    & g_d(D\oplus D') = g_d(D)g_d(D') \\
    & g_{dr}(D) = g_d(D),
  \end{align*}
  where $r\in \Z$ is square in $(\Z/N\Z)^\times$. 
\item[ii)]
  If $d$ is coprime to $|D|$, we  have
\begin{equation}\label{eq:frac_gauss_sum}
  \frac{g(D)}{g_d(D)} =\leg{d}{|D|}e\left(\frac{(d-1)\odd(D)}{8}\right).
\end{equation}
If $|D|$ is odd, the right-hand side of \eqref{eq:frac_gauss_sum} simplifies to
the quadratic character
\begin{equation}\label{eq:frac_gauss_char}
  \chi_{D}(d) = \leg{d}{|D|}.
\end{equation}
\end{enumerate}
  \end{lemma}
One way to proof equation \eqref{eq:frac_gauss_sum} is by the action of the finite Weil represntation on matrices $\gamma\in \Gamma$ with $\gamma \equiv \kzxz{a}{0}{0}{d}\bmod{N}$. McGraw, \cite{McG}, Lemma 4.6, and Scheithauer, \cite{Sch}, Prop. 4.2, calculated the Weil representation of such matrices $\gamma$ leading to identity \eqref{eq:frac_gauss_sum}.
As for equation \eqref{eq:frac_gauss_char}, Borcherds in \cite{Bo1}, Theorem 5.4, gave also a formula for the action of $\rho_L$ on these matrices $\gamma$. His formula is more general since it includes the case of discriminant forms of odd signature. Comparing Scheithauer's result with Borcherds' result for even signature gives the identity \eqref{eq:frac_gauss_char}. 

  The theory of Hecke operators for modular forms of this type was developed in \cite{BS}. In Theorem \ref{thm;action_hecke_op_notcoprime} of the present paper Hecke operators $T(M,\det(M)^{1/2})$ for matrices of the form $M=\kzxz{p^{-k}}{0}{0}{p^{-l}}$ with a prime number $p$, $(k,l)\in \Lambda$, play a main role.
  If $p$ is coprime to $N$, it can be defined by the action of the corresponding double coset in the classical way, see  Def. 4.1 of \cite{BS} for details. The main difficulty lies in the extension of $\rho_L$ to matrices of $\GL_2(\Q)$, which only works if the componentwise reduction modulo $N$ of the matrix in question yields an element in $\GL_2(\Z/N\Z)$ and if the determinant is a square modulo $N$.
  For a prime $p$ dividing $N$, the reduction modulo $N$ of $\kzxz{p^{-k}}{0}{0}{p^{-l}}$ does not lie in $\GL_2(\Z/N\Z)$. However, as explained in \cite{BS}, Chapter 5 ((5.1), (5.2) more specifically),  taking  \cite{St2}, Section 4, into account, it is possible to define a Hecke operator by
  setting
  \begin{equation}
   f\mid_{\kappa,L} T(m(p^{-k},p^{-l}), p^{-(k+l)/2}) = \det(m(p^{-k},p^{-l}))^{\kappa/2-1}\sum_{\gamma\in \Gamma\bs\Gamma m(p^{-k},p^{-l})\Gamma}f\mid_{\kappa,L}\gamma.
    \end{equation}
  Here, for any $\gamma = \delta m(p^{-k},p^{-l}) \delta'\in \Gamma m(p^{-k},p^{-l})\Gamma$ we put
  \begin{equation}\label{eq:weil_double_cosets}
  \rho_L^{-1}(\gamma) = \rho_L^{-1}(\delta')\rho_L^{-1}(m(p^{-k},p^{-l}))\rho_L^{-1}(\delta)
  \end{equation}
  and
  \begin{equation}\label{eq:weil_diagonal_matrix}
    \begin{split}
      \rho_L^{-1}(m(p^{-k},p^{-l}))\frake_\lambda &= \rho_L^{-1}(m(p^{-l},p^{-l}))\rho_L^{-1}(m(p^{l-k},1))\frake_\lambda \\
      &= \frac{g(D_p^\perp)}{g_{p^l}(D_p^\perp)}\frake_{p^{(l-k)/2}\lambda}, 
    \end{split}
  \end{equation}
  where $D_p^\perp$ means the orthogonal complement of $D_p$ in $D$.
  
  \section{The Weil representation on $\GL_2(\A)$}\label{subsec:weil_repr_adelic}
Let $(L, (\cdot,\cdot))$ be an even, non-degenerate lattice of type $(b^+,b^-)$ with even rank $m$, with dual lattice $L'$ and the quadratic module $D=L'/L$. 
We further  define $V=L\otimes \Q$ and let $H=O(V)$ be the orthogonal group over $\Q$ attached to $(V,(\cdot,\cdot))$. 
In this section we collect some well known facts on the Weil representation of $\SL_2(\A)\times H(\A)$, which is suited for our purposes in this paper. 
 Here we consider the Schr\"odinger model of the Weil representation $\omega = \prod_{p\le \infty}\omega_p$  on the space $S(V(\A))$ of Schwartz-Bruhat functions associated to the character
\begin{equation}\label{eq:standard_character}
  \psi = \prod_{p\le \infty}\psi_p: \A/\Q\rightarrow \C^{\times},\; x=(x_p)\mapsto \psi(x)=e^{2\pi i (-x_\infty + \sum_{p<\infty}x_p')},
\end{equation}
where $x_p'\in \Q/\Z$ is the principal part of $x_p$ and $V(\A)=V\otimes \A$.  Note that this character is the complex conjugate   of the standard additive character (see e. g. \cite{St}, \cite{BY} and \cite{KL}, Chapter 8).

A second goal of the present section is the extension of $\omega$ to a subgroup of $\GL_2(\A)$ in the spirit of the extension of the ``finite'' Weil representation $\rho_L$ in \cite{BS}.  

For $\mu\in D$  we define  $\varphi_\mu \in S(V(\A_f))$ with
\begin{align}\label{eq:familiy}
  \varphi_\mu = \mathbbm{1}_{\mu + \hat{L}} = \prod_{p<\infty} \varphi^{(\mu)}_p = \prod_{p<\infty}\mathbbm{1}_{\mu+L_p}. 
\end{align}
Here $L_p=L\otimes \Z_p$, which is the $p$-part of $\hat{L}=L\otimes \hat{\Z}$ with $\hat{\Z}=\prod_{p<\infty}\Z_p$, and $\mathbbm{1}_{\mu+L_p}$ is the characteristic function of $\mu+L_p$. Note that there is a close relation between the finite groups $L_p'/L_p$ and the $p$-groups $D_p$. In fact, these groups are isomorphic.  This  isomorphism additionally respects the quadratic forms of both groups see e. g. \cite{Ze}, Section 3, \cite{St}, Remark 3.2 or \cite{We1}, Theorem 4.30.   In the following, we identify these groups and use them interchangeably. 
As in \cite{BY}, we consider the $|D|$-dimensional subspace
\begin{equation}\label{eq:space_char_func}
  S_L=\bigoplus_{\mu\in D}\C\varphi_\mu \subset S(V(\A_f)). 
\end{equation}
It is known that the space in \eqref{eq:space_char_func} is stable under the action of the group $\SL_2(\Z)$ via the Weil representation $\omega_f$ (see e. g.  \cite{YY}, p. 3456). Also, the $L^2$ scalar  product $\langle \cdot,\cdot\rangle$  on $S_L\subset S(V(\A_f))$ simplifies to
\begin{equation}\label{def:L_2_S_L}
\langle \sum_{\mu\in D}F_\mu\varphi_\mu, \sum_{\mu\in D}F_\mu\varphi_\mu\rangle = \sum_{\mu\in D}|F_\mu|^2.
\end{equation}
Note that $D$ can be decomposed into $p$-groups $D = \bigoplus_{p\mid |D|}D_p\cong\bigoplus_{p\mid |D|}L_p'/L_p $. For almost all primes $p$ - those coprime to $|D|$ - $L_p$ is unimodular and thus $L_p'/L_p = \{0+L_p\}$. Therefore, we can write $D\cong \bigoplus_{p<\infty} L_p'/L_p$. On the level of the space $S_L$, this decomposition translates to the isomorphism
\begin{equation}\label{eq:local_decomp_S_L}
S_L\cong \bigotimes_{p< \infty} S_{L_p},\quad  \varphi_{\mu} \mapsto \bigotimes_{p<\infty}\varphi_p^{(\mu_p)},
\end{equation}
where $\mu= \sum_{p\mid |D|}\mu_p$ and $\varphi_p^{(\mu_p)} = \varphi_p^{(0)}$ for all primes $p$ coprime to $|D|$. The local Weil representation $\omega_p$ acts on the $p$-part $S_{L_p}$ of $S_L$, where
\begin{equation}\label{eq:local_S_L}
S_{L_p} =
\begin{cases}
  \bigoplus_{\mu\in L_p'/L_p}\C\varphi_p^{(\mu)}, & p\mid |D|,\\
  \C\varphi_p^{(0)}, & p\nmid |D|.
\end{cases}
  \end{equation}
We then have 
 \[
\omega_f(\gamma_f)\varphi_\mu = \bigotimes_{p<\infty}\omega_p(\gamma_p)\varphi_p^{(\mu_p)}.
\]

According to  \cite{St}, Lemma 3.4 and \cite{BY}, Proposition 2.5, the Weil representation $\omega_p$ can be described explicitly on the generators of $\SL_2(\Z_p)$ by
\begin{equation}\label{eq:weil_rep_explicit}
  \begin{split}
    &    \omega_p(n(b))\varphi_p^{(\mu_p)} = \psi_p(bq(\mu_p))\varphi_p^{(\mu_p)} \\
    & \omega_p(w)\varphi_p^{(\mu_p)} = \frac{\gamma_p(L_p'/L_p)}{|L_p'/L_p|^{1/2}}\sum_{\nu_p\in L_p'/L_p}\psi_p((\mu_p,\nu_p))\varphi_p^{(\nu_p)}\\
    & \omega_p(m(a))\varphi_p^{(\mu_p)} = \chi_{V,p}(a)\varphi_p^{(a^{-1}\mu_p)},
    \end{split}
\end{equation}
where $\gamma(L_p'/L_p)$ is the local Weil index and  $\chi_{V,p}(a) = (a, (-1)^{m/2}|D_p|)_p$ is the local Hilbert symbol. Evaluating the local Hilbert symbol gives
\begin{equation}\label{eq:local_hilbert_symbol}
  \chi_{V,p}(a) = \leg{a}{|L_p'/L_p|} = \chi_{D_p}(a),
\end{equation}
see e. g. \cite{Se}, Chapter III. 
These formulas imply that (see the proof Lemma 3.4 in \cite{St})
\begin{enumerate}
\item[i)]
  the local Weil representations $\omega_p$ is trivial if $p$ is coprime to $|D|$,
\item[ii)]
  if we identify $\C[D]$ with $S_L$ via $\frake_\mu\mapsto \varphi_\mu$, then $\omega_f$ coincides with the finite Weil representation $\rho_{L}$ in the following way 
  \begin{equation}\label{eq:adelic_weil_rerpr}
  \rho_L (\gamma) = \omega_f(\gamma_f),
  \end{equation}
  where $\gamma\in \SL_2(\Z)$ and $\gamma_f\in \SL_2(\widehat{\Z})$ is the projection of $\gamma$ into $\SL_2(\widehat{\Z})$. Note that by our choice of the character $\psi$, the relation \eqref{eq:adelic_weil_rerpr} differs from the one in \cite{BY}, (2.7), by conjugation.
  \end{enumerate}
The following lemma provides the action of $\omega_p$ for the lower triangular matrix $n\_(c)\in \SL_2(\Z_p)$:

\begin{lemma}\label{lem:local_weil_lower_triangular}
  \begin{enumerate}
    \item[i)]
  Let $c\in \Z_p^\times$. Then
  \begin{equation}\label{eq:local_weil_lower_triangular}
    \omega_p(n\_(c))\varphi_p^{(\mu_p)} = \frac{\gamma_p(L_p'/L_p)}{|L_p'/L_p|^{1/2}}\chi_{V,p}(-c)\psi_p(c^{-1}q(\mu_p))\sum_{\nu_p\in L_p'/L_p}\psi_p(c^{-1}q(\nu_p))\psi_p(-c(\mu_p,\nu_p))\varphi_p^{(\nu_p)}.
    \end{equation}
\item[ii)]
  Let $c\in p\Z_p$. If $L_p'/L_p$ is anisotropic, then
  \begin{equation}\label{eq:local_weil_lower_triangular_1}
    \omega_p(n\_(c))\varphi_p^{(\mu_p)} = \varphi_p^{(\mu_p)}. 
  \end{equation}
  \end{enumerate}
  \end{lemma}
  \begin{proof}

    $i)$:  By the Bruhat decomposition (see e. g. \cite{KL}, p. 69),
  \[
  n\_(c) = n(c^{-1})wn(c)m(-c).
  \]
  From this we infer that by means of \eqref{eq:weil_rep_explicit}
  \[
  \omega_p(n\_(c))\varphi_p^{(\mu_p)} = \frac{\gamma_p(L_p'/L_p)}{|L_p'/L_p|^{1/2}}\chi_{V,p}(-c)\psi_p(c^{-1}q(\mu_p))\sum_{\nu_p\in L_p'/L_p}\psi_p(c^{-1}q(\nu_p))\psi_p(-c(\mu_p,\nu_p))\varphi_p^{(\nu_p)}.
  \]

  $ii)$:   Since $c\in p\Z_p$ and the level of $L_p$ is $p$, $\omega_p$ acts trivially on $S_{L_p}$:
  \begin{align*}
    \omega_p(n\_(ca^{-1}))\varphi_p^{(\mu_p)} &= \omega_p(w)\omega_p(n(ca^{-1}))\omega_p(w^{-1})\varphi_p^{(\mu_p)} \\
    &=\omega_p(w)\omega_p(w^{-1})\varphi_p^{(\mu_p)} \\
    &=\varphi_p^{(\mu_p)}.
  \end{align*}
  For the second equation we used that $\omega_p(n(ca^{-1}))$ acts trivially on $S_{L_p}$.
\end{proof}

Via the extension of $\rho_{L}$ to a subgroup of $\calG(\Q)$ (see \cite{BS} and \cite{St2}, Section 4), it is possible to extend $\omega_f$ (cf.  \cite{We}, Def. 46) to the same group, into which $\SL_2(\A)$ can be embedded. To explain this extension process, we need some notation. Let $N= \prod_{i=1}^rp^{e_i}$, 
\[
\pi:\widehat{\Z}\rightarrow \prod_{p|N}\Z_p
\]
the projection onto the places $p\mid N$ and
\[
\pi_N: \prod_{p|N}\Z_p\rightarrow \Z/N\Z
\]
the composition of the canonical projection  of $\Z_{p_i}$ to $\Z/p^{e_i}\Z$ and the application of the Chinese remainder theorem. We further denote with 
\begin{equation}\label{eq:K_p}
\calK(p):=\left\{(M,r)\in \calK_p\times \Z_p^\times\;|\; \det(M)=r^2\right\}
\end{equation}
and $\calK' = \prod_{p<\infty}\calK(p)$. The following group can be found in \cite{BS}, (3.2):
\[
Q(N) = \{(M, r)\in \GL_2(\Z/N\Z)\times (\Z/N\Z)^\times\;|\; \det(M)\equiv r^2\bmod{N}\}.
  \]
  Applied to each component of the involved matrix, we obtain a sequence of homomorphisms
 \begin{align*}\label{eq:K_prime}
  \calK'\xrightarrow{(\Pi,\pi)}\prod_{p|N}\calK(p)\xrightarrow{(\Pi_N,\pi_N)} Q(N),
 \end{align*}
 where $\Pi$ and $\Pi_N$ denote the matrix valued counterparts of $\pi$ and $\pi_N$, respectively.
Note that we can embed $\SL_2(\Z_p)$ and $\{\kzxz{r}{0}{0}{r}\; |\; r\in (\Z_p)^\times\}$ into $\calK(p)$ homomorphically by the mappings $k_p \mapsto (k_p,1)$ and $\kzxz{r}{0}{0}{r}\mapsto (\kzxz{r}{0}{0}{r},r)$. 
For the rest of the paper, we omit the second component of elements of $\calK', \calK(p)$ or $Q(N)$ (if possible) and use the groups $\calK_p$ and $\calK(p)$ interchangeably.
 
 \begin{definition}\label{def:ext_adelic_weil_repr_coprime}  
   Let $k\in \calK'$. Then we define
   \begin{equation}\label{eq:weil_rep_global}
  \begin{split}
  \omega_f(k) &= \bigotimes_{p<\infty} \omega_p(k_p)\\
  &=\bigotimes_{p\nmid N}\omega_p(k_p)\bigotimes_{p|N}\omega_p(k_p)
  \end{split}
  \end{equation}
with $\omega_p(k_p) = \id_{S_{L_p}}$ for all primes $p\nmid N$ and  
\begin{equation}\label{eq:ext_adelic_weil_repr_coprime}
  \begin{split}
 \bigotimes_{p\mid N}\omega_p(k_p)= \rho_{L}(\Pi_N((k_p)_{p\mid N})).
    \end{split}
\end{equation}
Here by $(k_p)_{p\mid N}$ we mean the tuple of all components $k_p\in \calK(p)$ of $k$ belonging to the primes $p$ dividing $N$. 
Combining \eqref{eq:weil_rep_global} and \eqref{eq:ext_adelic_weil_repr_coprime} we set
\begin{equation}\label{def:global_weil_rep}
\omega_f(k) = \rho_{L}((\Pi_N\circ \Pi)(k)).
\end{equation}
\end{definition}

 Note that   Definition \ref{def:ext_adelic_weil_repr_coprime} is compatible with \eqref{eq:adelic_weil_rerpr}. For,  if we take $k\in \SL_2(\widehat{\Z})$ as the projection of some $\gamma\in \SL_2(\Z)$, we find $(\Pi_N\circ \Pi)(k) = \Pi_N(\gamma)\in \SL_2(\Z/N\Z)$ and
\begin{equation}\label{eq:weil_on_SL_2}
 \omega_f(k) = \rho_{L}(\gamma)
\end{equation}
since $\rho_L$ factors through $\SL_2(\Z/N\Z)$.
As a special case, Definition \ref{def:ext_adelic_weil_repr_coprime} comprises the extension of the local Weil representation $\omega_p$  from $\SL_2(\Z_p)$ to $\calK(p)$:
 \begin{definition}\label{def:local_weil_repr_extended}
   Assume that the level $N$ is equal to a prime $p$ and let $k_p\in\calK(p)$ be embedded into $\calK'$ by $\iota_p(k_p)$. Then  we define
   \begin{equation}\label{eq:local_weil_repr_extended}
     \omega_p(k_p) = \omega_f(\iota_p(k_p)) = \rho_L(\Pi_p(k_p)).
   \end{equation}
\end{definition}
 The following formulas for $\omega_p$ will be used frequently. 
 \begin{remark}\label{def:ext_local_weil_repr}
   Let $k_p\in \calK_p$ with $\det(k_p)= t^2\in (\Z_p)^\times$. Then due to Definition \ref{def:local_weil_repr_extended} and
   \[
   \rho_L(\kzxz{r}{0}{0}{r},r)\frake_\lambda = \frac{g(D)}{g_r(D)}\frake_\lambda
   \]
for $(\kzxz{r}{0}{0}{r},r)\in \calQ(N)$  (cf. \cite{BS}, (3.5)), we find
 \begin{equation}\label{eq:weil_repr_scalar_matrix}
  \begin{split}
    \omega_p(\kzxz{t}{0}{0}{t})\varphi_p^{(\lambda_p)} &= \rho_L(\Pi_p\kzxz{t}{0}{0}{t})\varphi_p^{(\lambda_p)}\\
    &= \frac{g(D_p)}{g_t(D_p)}\varphi_p^{(\lambda_p)} \\
 &= \chi_{D_p}(t)\varphi_p^{(\lambda_p)}
  \end{split}
 \end{equation}
 and
 \begin{equation}\label{eq:local_weil_repr}
\omega_p((k_p,t)) = \omega_p(\kzxz{t}{0}{0}{t}) \omega_p(\kzxz{t^{-1}}{0}{0}{t^{-1}}k_p),
 \end{equation}
 where $\kzxz{t^{-1}}{0}{0}{t^{-1}}k_p$ is an element of $\SL_2(\Z_p)$.
For the computations \eqref{eq:weil_repr_scalar_matrix} we identified $t$ with $\pi_p(t)$ and $\varphi_p^{(\lambda_p)}$ with $\frake_\lambda$ (regarding the right-hand side of the first equation). 
 \end{remark}

Finally, we need  to define  the local Weil representation $\omega_p$ on  double cosets of the form $\calK(p) m(p^{-k},p^{-l})\calK(p)$ where $p$ divides the level of $L$. These matrices play an important role in  Theorem \ref{thm;action_hecke_op_notcoprime}. Here  $m(p^{-k},p^{-l})\in \calM_p$ with $(k,l)\in \Lambda_+$. Note that we cannot proceed as in the definitions before since $\Pi_p$ is not well defined in this case. But we can mirror the corresponding process in \cite{BS}, Chapter 5. For the definition of ``finite'' Weil representation $\rho_L$ on $m(p^{-k},p^{-l})$ we refer to Section \ref{sec:preliminaries} or in more detail to \cite{St2}, Section 4 ((4.7)-(4.10)). 

\begin{definition}\label{def:local_weil_double_coset}
  Let  $m(p^{-k},p^{-l})\in \calM_p$ with $(k,l)\in \Lambda$. 
  \begin{enumerate}
  \item[i)]
    Then we set
    \begin{equation}\label{eq:local_weil_double_coset_1}
      \begin{split}
      \omega_p(m(p^{-k},p^{-l}))^{-1}\varphi_p^{(\lambda_p)} &= \rho_L(m(p^{-k},p^{-l}))^{-1}\varphi_p^{(\lambda_p)}\\
      &= \rho_L(m(p^l,p^l))^{-1}\rho_L(m(p^{l-k},1))^{-1}\varphi_p^{(\lambda_p)} \\
      &= \varphi_p^{(p^{(l-k)/2}\lambda_p)},
        \end{split}
    \end{equation}
    where we used for the last equation that according to \cite{St2}, (4.7), (4.8), $\rho_L(m(p^l,p^l))$ acts by multiplication with $\frac{g(D_p^\perp)}{g_{p^l}(D_p^\perp)}$. But, this quotient is by definition in this special case trivial because $D_p^\perp$ is trivial. 
  \item[ii)]
    For $\delta = \gamma m(p^{-k},p^{-l})\gamma'\in\calK(p)m(p^{-k},p^{-l})\calK(p)$ we define
    \begin{equation}\label{eq:local_weil_double_coset_2}
      \begin{split}
        \omega_p(\delta)^{-1}\varphi_p^{(\lambda_p)} &= \omega_p(\gamma')^{-1}\omega_p(m(p^{-k},p^{-l}))^{-1}\omega_p(\gamma)^{-1}\varphi_p^{(\lambda_p)} \\
     &  = \omega _p(\gamma')^{-1}\rho_L(m(p^{-k},p^{-l}))^{-1}\omega_p(\gamma)^{-1}\varphi_p^{(\lambda_p)}.
      \end{split}
      \end{equation}
    \end{enumerate}
  \end{definition}

\begin{remark}
The definition of $\omega_p$ on $\calK(p)m(p^{-k},p^{-l})\calK(p)$ is independent of the choice of the representatives. This follows from \cite{BS}, Prop. 5.1, for double cosets of the form $\SL_2(\Z_p)m(p^{-k},p^{-l})\SL_2(\Z_p)$ and \eqref{eq:weil_on_SL_2}. Since the action of $\omega_p$ on $\calK(p)$ differs from that on $\SL_2(\Z_p)$ only by the action of scalar matrices, i. e. by multiplication with a character (see Definition \ref{def:ext_local_weil_repr}), we obtain the same result for double cosets of the form $\calK(p)m(p^{-k},p^{-l})\calK(p)$. 
\end{remark}


\section{The Hecke algebra $\calH(\calQ_p//\calK_p, \omega_p)$}\label{subsec:hecke_slgebras}
In this section we will describe the structure of the local vector valued spherical Hecke algebra $\calH(\calQ_p//\calK_p,\omega_p)$ associated to the pair of groups $(\calQ_p, \calK_p)$ and the local Weil representation $\omega_p$. For each prime $p$ we will introduce a Satake map, which allows us to understand the structure of this Hecke algebra. For primes $p\nmid |D|$ these Hecke algebras are isomorphic to the classical scalar valued Hecke algebras defined by the same groups. These are well understood thanks to the classical Satake map. If  $p$ divides $|D|$, the algebras $\calH(\calQ_p//\calK_p,\omega_p)$ are considerably more complicated because $\omega_p$ is non-trivial. However, under certain restrictions for $D_p$, we will define a modified Satake map, which maps $\calH(\calQ_p//\calK_p,\omega_p)$ to a simpler algebra, whose structure can be easier determined.


The following general facts about spherical Hecke algebras can be found in many places, among them \cite{BK}, chapter 4, \cite{Ho} and \cite{Mu}. 


  \begin{definition}\label{def:spherical_hecke_algebra}
Let $G$ be a locally compact group $G$, $K$ an open compact subgroup and $\rho: K\rightarrow \GL(V)$ a representation of $K$. 
    The Hecke algebra $\calH(G//K,\rho)$ of $\rho$-spherical functions  is the set of functions $f:G\rightarrow \text{End}(V)$ which 
    \begin{enumerate}
    \item[i)]
      are compactly supported modulo $K$, i. e.  each $f$ vanishes outside finitely many double cosets $K g K$ and satisfy
    \item[ii)]
\[
f(k_1 g k_2) = \rho(k_1)\circ f(g)\circ\rho(k_2) \text{ for all } k_1, k_2 \in K  \text{ and all } g\in G.
\]
    \end{enumerate}
    Since each element $f$ of $\calH(G//K,\rho)$ is of the form
    \[
    f(g)= \sum_{i=1}^n a_if_i(g), 
    \]
    where $a_i\in \C$ and  $f_i$ is an element of the subspace of functions of $\calH(G//K,\rho)$, which vanish outside $Kg_iK$, the whole algebra is generated by the functions $f_i$. 
    Similarly, we denote by $\calH(G//K)$ the set of functions $f:G\rightarrow \C$, which are compactly supported modulo $K$ and $K$-bi-invariant, i. e. $f(k_1 g k_2) = f(g)$ for all $k_1, k_2\in K$ and all $g\in G$. We call $\calH(G//K)$ also a spherical Hecke algebra. 
\end{definition}

  It is well known that $\calH(G//K, \rho)$ is an associative $\C$-algebra with respect to convolution
  \begin{equation}\label{eq:convolution}
  (f_1\ast f_2)(g) = \int_{G}f_1(h)\circ f_2(h^{-1}g)dh,
  \end{equation}
  where $dh$ is the standard Haar measure on $G$ normalized by $\int_{K}dh = 1$.
  Provided that $G/K$ is countable, we may write
  \[
  (f_1\ast f_2)(g) = \sum_{h\in G/K}f_1(h)\circ f_2(h^{-1}g).
  \]

  In order to determine the structure of $\calH(G//K, \rho)$, in view of the remarks before, it is useful to study the space of functions in this Hecke algebra, which vanish outside a single double coset $K g K$. It can be described in terms of intertwining operators of $\rho$ associated with $g$. To state the corresponding result, we fix some notation. For $g\in G$ we mean by $K^g$ the group $gKg^{-1}$ and write $\rho_g$ for the representation $h\mapsto \;\rho_g(h)= \rho(g^{-1}hg)$ of $K^g\cap K$. As usual,
  \begin{align*}
  \Hom_{K\cap K^g}(\rho,\rho_g) =    \{F:V\rightarrow V\; |\; F \text{ is linear and } F\circ \rho_g(h) = \rho(h)\circ F \text{ for all } h\in K\cap K^g\}.
  \end{align*}
  Then we have

  \begin{lemma}\label{lem:double_cose_hecke_algebra}
   Let $g\in G$. The subspace of $\calH(G//K,\rho)$ consisting of functions supported on $K g K$, is isomorphic to $\Hom_{K \cap K^g}(\rho, \;\rho_g)$.
    \end{lemma}
  \begin{proof}
    The assertion is well known (see e. g. \cite{BK}, Chapter 4). Nevertheless, for later purposes,  we indicate a proof by giving the maps of the claimed isomorphism (without further explanation).

    If $f \in \calH(G//K,\rho)$ with $f(g)\not=0$ supported on $KgK$, then it easily checked that $f(g)\in \Hom_{K \cap K^g}(\rho, \;\rho_g)$ (non-zero). On the other hand, if $0\not=F \in \Hom_{K \cap K^g}(\rho, \;\rho_g)$, we put $f(g) = F$ and $f(k_1 g k_2) = \rho(k_1)\circ f\circ \rho(k_2)$ and obtain thereby an element of the above stated subspace of $\calH(G//K,\rho)$. 
    \end{proof}
  The following Lemma ensures that the groups $\calQ_p$ and $\calK_p$ meet the conditions of Definition \ref{def:spherical_hecke_algebra}. It might be known. Since  I have not found it in the literature, I state it here and add a short proof.
\begin{lemma}\label{lem:properties_groups}
  \begin{enumerate}
  \item[i)]
    The group $\calK_p$ is an open compact subgroup of $\calQ_p$.
  \item[ii)]
    The group $\calQ_p$ is a locally compact subgroup of $\GL_2(\Q_p)$.
    \end{enumerate}
\end{lemma}
\begin{proof}
  It is well known that $\GL_2(\Q_p)$ is locally compact.
  By Lemma 8, I.3, of \cite{Ch}, it follows that $\calQ_p$ is also locally compact.  
  By \cite{Ka}, Thm. 2.15, we know that $(\Z_p^\times)^2$
  is an open subgroup in $\Z_p^\times$
  Therefore, $\calK_p$ is an open subgroup in $\GL_2(\Z_p)$, which implies that is is also closed (see \cite{HW}, Thm. 5.5). As $\calQ_p\cap \GL_2(\Z_p)=\calK_p$, we find that $\calK_p$ is a compact subgroup of $\calQ_p$. 
\end{proof}
Note that there is analogue of the Cartan decomposition for the pair $(\calQ_p, \calK_p)$.
\begin{lemma}\label{lem:cartan_decomp}
 The group $\calQ_p$ can be written as a disjoint union of $\calK_p$-double cosets:
\[
\calQ_p = \bigcup_{\substack{k\le l\\ k+l\in 2\Z}}\calK_p m(p^k, p^l)\calK_p.
\]  
\end{lemma}
\begin{proof}
  The proof is the same as for the Cartan decomposition for $\GL_2(\Q_p)$, see e. g. \cite{Mu}, p. 17. In the quoted proof the matrix $g=\kzxz{a}{b}{c}{d}\in \GL_2(\Q_p)$ with $a\not=0$ and  $|a|_p \ge \max\{|b|_p,|c|_p,|d|_p\}$ is transformed into
  \[
  m(a,d-a^{-1}bc)= k_3gk_4,
  \]
  where $k_3 = n\_(-a^{-1}c), k_4=n(-a^{-1}b)\in \calK_p$. If we assume that $\det(g) = p^{2x}y^2$ and $a=p^ks$, then $d-a^{-1}bc = p^{2x-k}y^2s^{-1}, y, s\in \Z_p^\times,$ and
  \[
  m(p^k,p^{2x-k}) = n\_(-a^{-1}c)g\; n(-a^{-1}b)m(s)m(1,y^2).
  \]
  Therefore, all used transformation matrices are contained in $\calK_p$. We have a similar decomposition if $d\not=0$, see \cite{KL}, p. 208.

  Also, note that any two double cosets $\calK_p g_1 \calK_p$, $\calK_p g_2\calK_p$  are disjoint since otherwise the  double cosets $\GL_2(\Z_p)g_1\GL_2(\Z_p)$, $\GL_2(\Z_p)g_2\GL_2(\Z_p)$ would not be disjoint, contradicting the Cartan decomposition for $\GL_2(\Q_p)$.  
\end{proof}

We also have an analogue of the Iwasawa decomposition of $\GL_2(\Q_p)$.
\begin{lemma}\label{lem:iwasaw_decomp}
  \begin{equation}\label{eq:iwasawa_decomp}
    \calQ_p = \calM_pN(\Q_p)\calK_p.
    \end{equation}
  \end{lemma}
  \begin{proof}
    This follows immediately from the Iwasawa decomposition for $\GL_2(\Q_p)$ by the intersection with $\calQ_p$ on both sides.
    \end{proof}

  As already noted, there are two cases to consider regarding the  structure of $\calH(\calQ_p//\calK_p, \omega_p)$. It depends on whether $p$ divides $|D|$ or not. 
  In both cases we will determine a set of generators with the help of Lemma \ref{lem:double_cose_hecke_algebra}. Afterwards, we will define a  {\it Satake map}. If $p$ divides $|D|$, we will show - under the restriction that $D_p$ is {\it anisotropic} - that $\calH(\calQ_p//\calK_p, \omega_p)$ is isomorphic to a subalgebra of the spherical Hecke algebra $\calH(\calM_p//\calD_p, \omega_{p_{|S_{L_p}^{N(\Z_p)}}})$, where
\[
S_{L_p}^{N(\Z_p)}=\{\varphi\in S_{L_p}\; |\; \omega_p(n)(\varphi) = \varphi \text{ for all } n\in N(\Z_p)\}.
\] 
If $(p,|D|)=1$, $\calH(\calQ_p//\calK_p, \omega_p)$ is isomorphic to the classical spherical Hecke algebra \newline $\calH(\calQ_p//\calK_p)$, whose structure is well known.
We start with the discussion of the first mentioned case and consider the latter case subsequently.  

To describe the structure of  $\Hom_{\calK_p\cap \calK_p^g}(\omega_p, \rho_g)$, we need the decomposition of $S_{L_p}$ into irreducible submodules. This decomposition is well known, see for example \cite{NW}, Satz 2, Satz 4 and pages 521-522. We recall those parts relevant for the next lemma. We denote with $\Aut(D_p)$ the group of all automorphisms $\varepsilon$ of $D_p$ satisfying $q(\varepsilon(x)) = q(x)$ for all $x\in D$. Let further $U$ be a subgroup of $\Aut(D_p)$, which is determined for all possible cases of $D_p$ in \cite{NW}, Section 2,  and $\widehat{U}$ the dual group of $U$. It turns out that most of the {\it primitive} characters in $\widehat{U}$ give rise to an irreducible representation. The definition of a primitive character can be found on page 491 in \cite{NW}. 
We have to distinguish between the two possible anisotropic quadratic modules. For the case $\calA_p^t\oplus\calA_p^1$  Nobs and Wolfart proved the following decomposition of the space $S_{L_p}$ with respect to the Weil representation
\begin{equation}\label{eq:decomp_N_1}
  S_{L_p}\cong S_{L_p}(\chi_1)\bigoplus_{\substack{\chi\in \widehat{U} \text{ primitive}\\\chi^2\not=1}}S_{L_p}(\chi)\oplus \left(S_{M_p}(1,-)\oplus S_{M_p}(t,-)\right), 
\end{equation}
where $\chi_1=1$ means the trivial character and 
\begin{equation}\label{eq:decomp_N_1_spaces}
  \begin{split}
    &  S_{L_p}(\chi) = \left\{f\in S_{L_p}\;|\; f(\varepsilon x) = \chi(\varepsilon)f(x) \text{ for all } x\in \calA_p^t\oplus\calA_p^1 \text{ and all } \varepsilon \in U\right\},\\
    & S_{M_p}(t,-) = \left\{f\in S_{M_p}\; |\; f(-x) = -f(x) \text{ for all } x\in \calA_p^t \right\},
  \end{split}
   \end{equation}
$M_p$ being a $p$-adic lattice with $M_p'/M_p\cong \calA_p^t$. The space $S_{M_p}(1,-)$ is defined the same way by simply replacing $t$ with $1$. We write
\begin{equation}\label{eq:decomp_N_1_element}
f = f_1+\sum_{\substack{\chi\in \widehat{U} \text{ primitiv}\\\chi^2\not=1}}f_\chi+ f_++f_-
\end{equation}
for an element in $S_{L_p}$ with respect to \eqref{eq:decomp_N_1_spaces}. 
It is shown in \cite{NW} that $S_{L_p}(\chi_1)$ and $S_{L_p}(\chi_2)$ are isomorphic if and only if $\chi_1 = \chi_2$ or $\chi_1= \overline{\chi_2}$. The remaining quadratic modules in \eqref{eq:decomp_N_1} are  (pairwise) not isomorphic. The isomorphism between $S_{L_p}(\chi)$ and $S_{L_p}(\overline{\chi})$ is given explicitly in terms of the generators of $S_{L_p}(\chi)$:
A generator
\begin{equation}\label{eq:intertwining_op_conj}
f_{\mu_p}(\chi)= \sum_{\varepsilon\in U}\chi(\varepsilon)\varphi_p^{\varepsilon(\mu_p)}
\end{equation}
is mapped to $f_{\overline{\mu_p}}(\overline{\chi})$, where $\overline{\mu_p}$ is $(a+b,-b)$ for $\mu_p = (a,b)$. We denote this intertwining operator by $T^{\overline{\chi}}$. 

\subsection{The case of primes $p$ dividing $|D|$}\label{subsec:Hecke_algebra_ncoprime}
\begin{lemma}\label{lem:struct_hecke_double_coset_1}
 Let $D_p$ be an anisotropic discriminant form and $g=m(p^k,p^l)\in \calQ_p$ with $(k,l)\in \Lambda$. Put $\rho_g = (\omega_p)_g$. 
  \begin{enumerate}
    \item[i)]
      If  $k < l$, then the space  $\Hom_{\calK_p\cap \calK_p^g}(\omega_p, \rho_g)$ is generated by the map
      \begin{equation}\label{eq:intertwiner_1}
  T(k,l): S_{L_p}\rightarrow S_{L_p}^{N(\Z_p)}, \quad \varphi_p^{(\mu_p)} \mapsto T(k,l)(\varphi_p^{(\mu_p)}) = \varphi_p^{(p^{(l-k)/2}\mu_p)}= \varphi_p^{(0)}.
  \end{equation}
\item[ii)]
  If  $D_p\cong \calA_p^t\oplus \calA_p^1$, then $S_{L_p}$ decomposes into the irreducible submodules $S_{L_p}(\chi_1)$, $S_{L_p}(\chi)$, $S_{M_p}(1,-)$ and $S_{M_p}(t,-)$. 
  For $k=l$ the space $\Hom_{\calK_p\cap \calK_p^g}(\omega_p, \rho_g)$ is then generated by the maps $T(k,k)^{\chi}$,  $T(k,k)^{\overline{\chi}}$, $T(k,k)^{\chi_1}$, $T(k,k)^+$ and $T(k,k)^-$ where
  \begin{equation}\label{eq:intertwiner_k_equal_l}
    \begin{split}
     & T(k,k)^{\chi}(f) =
      f_\chi, \quad
     T(k,k)^{\overline{\chi}}(f) =
     T^{\overline{\chi}}(f_\chi), \\
      & T(k,k)^{+}(f) =
     f_+,\quad T(k,k)^-(f) = f_- \text{ and }\\
      & T(k,k)^{\chi_1}(f) = f_1 
      \end{split}
  \end{equation}
and $f$ is  an element in $S_{L_p}$ as in \eqref{eq:decomp_N_1_element}. 
\item[iii)]
  If  $D_p\cong \calA_p^t$, then $S_{L_p}$ decomposes into the irreducible submodules $S_{L_p}(t,+)$ and $S_{L_p}(t,-)$, where these spaces are defined as in \eqref{eq:decomp_N_1_spaces} with $M_p$ replaced by $L_p$.
  For $k=l$ the space $\Hom_{\calK_p\cap \calK_p^g}(\omega_p, \rho_g)$ is then generated by the two  maps $T(k,k)^{+}$ and $T(k,k)^-$, where
  \begin{equation}\label{eq:intertwiner_k_equal_l_one_dim}
 T(k,k)^{+}(f) =
     f_+ \text{ and } T(k,k)^-(f) = f_-.
  \end{equation}
  \end{enumerate}
\end{lemma}
\begin{proof}
  In light of Lemma \ref{lem:cartan_decomp}, it clearly suffices to choose $g=m(p^k,p^l)$ with $(k,l)\in\Lambda$.
  
  i)  First, note that

  \[
  m(p^k,p^l)\kzxz{a}{b}{c}{d}m(p^k,p^l)^{-1} = \kzxz{a}{bp^{k-l}}{cp^{l-k}}{d}.
  \]
  In particular, for $\kzxz{a}{b}{c}{d} = n(p^{l-k})$ we find that $n(1)$ is an element of $\calK_p\cap \calK_p^g$. Thus, for any $F\in \Hom_{\calK_p\cap \calK_p^g}(\omega_p, \rho_g)$ the equation
  \begin{align*}
  F\circ \omega_p(m(p^k,p^l)^{-1}n(1)m(p^k,p^l)) = \omega_p(n(1))\circ F
   & \Longleftrightarrow \omega_p(n(1))\circ F = F\circ \omega_p(n(p^{l-k}))
      \\
      &\Longleftrightarrow \omega_p(n(1))\circ F = F
 \end{align*}
  must hold.  For the last equivalence we have used that the level of $L_p$ is $p$. It follows that the image of $F$ is a subset of $S_{L_p}^{N(\Z_p)}$. Since the identity
  $\omega_p(n(b))\varphi_p^{(\gamma)} = \varphi_p^{(\gamma)}$ holds for all $b\in \Z_p^\times$ if and only if $\gamma$ is isotropic and $D_p$ anisotropic, we can conclude that $S_{L_p}^{N(\Z_p)} = \C\varphi_p^{(0)}$. Therefore, $F$ is a scalar multiple of the map $\varphi_p^{(\lambda_p)}\mapsto \varphi_p^{(0)}$ and has the claimed form.

  ii) The decomposition of $S_{L_p}$ into irreducible submodules is well known. For the quadratic module $\calA_p^t$ see for example \cite{NW}, Theorem 4.  The case of the quadratic module $\calA_p^t\oplus \calA_p^1$ is treated in \cite{NW}, Theorem 2 and Section 9, p. 521-522.  In the case $k=l$ the equation $F\circ \rho_g(h) = \omega_p(h)\circ F$ simplifies to $F\circ \omega_p(h) = \omega_p(h)\circ F$ for all $h\in \calK_p$. Thus, $F$ is an intertwining operator for $\omega_p$. The structure of the space of intertwining operators can be found in books about representation theory, cf. \cite{JL}, Chapter 11.  
  \end{proof}


In view of Lemma \ref{lem:double_cose_hecke_algebra} and Lemma \ref{lem:struct_hecke_double_coset_1}, the following corollary is immediate. 

\begin{corollary}\label{cor:struct_hecke_algebra}
  Let $p$ be a prime dividing $|D|$,  $D_p$ anisotropic and $(k,l)\in \Lambda$.  Then the Hecke algebra $\calH(\calQ_p//\calK_p, \omega_p)$ is generated by the following elements:
  \begin{enumerate}
  \item[i)]
    For $k<l$
    \begin{equation}\label{eq:gen_k_less_l}
      T_{k,l}(k_1 m(p^{k},p^l) k_2) = \omega_p(k_1)\circ T(k,l)\circ \omega_p(k_2),
    \end{equation}
    where $T_{k,l}$ is only supported on $\calK_pm(p^k,p^l)\calK_p$. Here  $T(k,l)$ is the intertwining operator specified in Lemma \ref{lem:struct_hecke_double_coset_1}, i).
  \item[ii)]
    For $k=l$ 
    \begin{equation}\label{eq:gen_k_equal_l}
      T_{k,k}(k_1 m(p^k, p^k) k_2) = \omega_p(k_1)\circ T(k,k)\circ \omega_p(k_2),
    \end{equation}
    where $\supp(T_{k,k}) = \calK_pm(p^k,p^k)\calK_p$ and $T(k,k)$ is  one of the operators
\[
T(k,k)^\chi, T(k,k)^{\overline{\chi}}, T(k,k)^{+}, T(k,k)^- \text{ or } T(k,k)^{\chi_1}
\]
given in Lemma \ref{lem:struct_hecke_double_coset_1}, ii) and iii).
  \end{enumerate}
\end{corollary}

The following theorem investigates the structure of the Hecke algebra $\calH(\calM_p//\calD_p, \omega_{p_{|S_{L_p}^{N(\Z_p)}}})$ assuming $D_p$ is anisotropic. 

\begin{theorem}\label{rem:hecke_algebra_mp}
Let $p$ be an odd prime dividing $|D|$ and $D_p$ anisotropic. 
  \begin{enumerate}
  \item[i)]
    Then
    \begin{equation}
      \begin{split}
      \omega_p(m(t_1,t_2))\varphi_p^{(0)} &=
      \begin{cases}
        \leg{t_1}{|D_p|}\varphi_p^{(0)}, & |D_p|=p,\\
        \varphi_p^{(0)}, & |D_p|=p^2
      \end{cases} \\
      &= \chi_{D_p}(t_1)\varphi_p^{(0)}
      \end{split}
    \end{equation}
    for all $m(t_1,t_2)\in \calD_p$.
  \item[ii)]
     Then $S_{L_p}^{N(\Z_p)}$ is equal to $\C\varphi_p^{(0)}$ and  the Hecke algebra 
    $\calH(\calM_p//\calD_p, \omega_{p_{|S_{L_p}^{N(\Z_p)}}})$ isomorphic 
     to the scalar valued Hecke algebra $\calH(\calM_p//\calD_p)$.

  \end{enumerate}
\end{theorem}
\begin{proof}
  As already mentioned in Section \ref{sec:preliminaries}, the order of an anisotropic quadratic module $D_p$ is either  $p^2$ or  $p$.
  
  i) Let $m(t_1,t_2)\in \calD_p$ with $\det(m(t_1,t_2))=t^2\in (\Z_p^\times)^2$. Then by \eqref{eq:local_weil_repr}
\begin{align*}
  \omega_p(m(t_1,t_2))\varphi_p^{(0)} &= \omega_p(m(t,t))\omega_p(m(t^{-1}t_1,t^{-1}t_2))\varphi_p^{(0)} \\
  &=\leg{t}{|D_p|}\leg{t^{-1}t_1}{|D_p|}\varphi_p^{(0)},
\end{align*}
where we have used  \eqref{eq:local_hilbert_symbol} and \eqref{eq:weil_repr_scalar_matrix}.
The claimed result now follows. 

For ii)
we first note that if  $D_p$ is anisotropic, the space $S_{L_p}^{N(\Z_p)}$ is equal to $\C\varphi_p^{(0)}$ and thus one-dimensional. From  i) we know that $\calD_p$ acts via  $\omega_p$ on $S_{L_p}^{N(\Z_p)}$ by multiplication with the quadratic character $\chi_{D_p}$. 
Consequently,
\[
\calH(\calM_p//\calD_p, \omega_{p_{|S_{L_p}^{N(\Z_p)}}}) = \calH(\calM_p//\calD_p, \chi_{D_p}),
\]
where the latter Hecke algebra is meant in the sense of Definition \ref{def:spherical_hecke_algebra} with the one-dimensional representation $\rho = \chi_{D_p}$. 
The structure of the latter algebra was discussed in \cite{Ho}, Remark 5.1. It was stated there that $\calH(\calM_p//\calD_p, \chi_{D_p})$ is isomorphic  to the usual spherical algebra $\calH(\calM_p//\calD_p)$. 
To state this isomorphism explicitly, we specify a set of generators for the former algebra. It is generated by elements of the form 
\[
T_{k,l}(m(t_1,t_2)m(p^k,p^l)m(s_1,s_2)) = \chi_{D_p}(m(t_1,t_2))\circ T(k,l)\circ \chi_{D_p}(m(s_1,s_2)) 
\]
with
\[
T(k,l) = \mathbbm{1}_{\calD_pm(p^k,p^l)\calD_p}\id_{S_{L_p}^{N(\Z_p)}}. 
\]
The isomorphism  is then given by
\begin{equation}\label{eq:isom_hecke_algebra_mp}
I_{\chi_{D_p}}: T_{k,l} =\mathbbm{1}_{\calD_pm(p^k, p^l)\calD_p}\cdot\id_{S_{L_p}^{N(\Z_p)}}\mapsto \chi_{D_p}T_{k,l},
\end{equation}
where have extended $\chi_{D_p}$ trivially to a quasi-character on the whole group $\calM_p$ (see also \cite{Ho}). 
\end{proof}


We now define the before mentioned  Satake map to further clarify the structure of $\calH(\calQ_p//\calK_p, \omega_p)$ and to connect it to the algebra $\calH(\calM_p//\calD_p)$: 
\begin{equation}\label{def:satake_map}
  \begin{split}
    & \calS: \calH(\calQ_p//\calK_p, \omega_p) \rightarrow \calH(\calM_p//\calD_p,\omega_p{{{_{|S_{L_p}^{N(\Z_p)}}}}} ), \\
  &  T\mapsto \left(m\mapsto \delta(m)^{1/2}\sum_{n\in N(\Q_p)/N(\Z_p)}T(mn)_{|S_{L_p}^{N(\Z_p)}}\right).
  \end{split}
\end{equation}

\begin{remark}\label{rem:satake_map}
  \begin{enumerate}
  \item[i)]
   Note that this definition is analogous to the one given by Herzig (\cite{He}) over a field in characteristic $p$. The modulus character $\delta\kzxz{m_1}{0}{0}{m_2} = \left|\frac{m_1}{m_2}\right|_p$ is also part of the classical  Satake map (see e. g. \cite{De}, Chap. 8), where it ensures that the image of the Satake map is invariant under the natural action of the Weyl group. Herzig omitted the modulus character in his definition of the Satake map  as it does not produce the invariance under the action of the Weyl group. Nevertheless, we keep it in the definition of $\calS$ since it indeed does share the property of invariance under the Weyl group. 
  \item[ii)]
    With the same arguments as after the statement of Theorem 1.2 and as in its proof  (Step 0) in \cite{He}, it can be proved that $\calS$ is a well defined map. To prove that $\calS$ is a $\C$-algebra homomorphism all calculations of Step 2 in the proof of Theorem 1.2 in \cite{He} remain valid in our situation. 
    \end{enumerate}
  \end{remark}
As is shown in Lemma \ref{lem:struct_hecke_double_coset_1}  and Corollary \ref{cor:struct_hecke_algebra}, the space of maps in $\calH(\calQ_p//\calK_p,\omega_p)$ with support equal to $\calK_pm(p^k,p^k)\calK_p$ is two-dimensional or of higher dimension if $\omega_p$ is considered on the whole space $S_{L_p}$. If we restrict ourselves to an irreducible subspace of $S_{L_p}$,  the before mentioned space is one-dimensional. On the one hand, this condition  would guarantee that the Satake map \eqref{def:satake_map} is indeed an isomorphism (without it, \eqref{def:satake_map} is not even injective, as is easily checked). On the other hand, it is to restrictive for our purposes.  So, in order to obtain an isomorphism between $\calH(\calQ_p//\calK_p, \omega_p)$ and a subalgebra of  $\calH(\calM_p//\calD_p, \omega_{p_{|S_{L_p}^{N(\Z_p)}}})$ via \eqref{def:satake_map}, we restrict $\calH(\calQ_p//\calK_p, \omega_p)$ to a subalgebra where the space of maps supported on the double coset $\calK_pm(p^k,p^k)\calK_p$ is replaced with the  subspace generated by the operator
\begin{equation}\label{eq:T_k}
  \begin{split}
    &  T_{k}(k_1 m(p^k, p^l) k_2) = \omega_p(k_1)\circ T(k)\circ \omega_p(k_2) \text{ with } \\
    & T(k)= \begin{cases}
      T^{\chi_1}(k,k)+\sum_{\substack{\chi\in \widehat{U} \text{ primitiv}\\\chi^2\not=1}}T^\chi(k,k)+ T^+(k,k)+T^-(k,k),  & \text{ for } \calA_p^t\oplus \calA_p^1\\
      T^+(k,k)+T^-(k,k), & \text{ for } \calA_p^t\\
    \end{cases}\\
& = \id_{S_{L_p}}.  
  \end{split}
\end{equation}
It will turn out that $T_k$ is compatible with the Hecke operator $T(m(p^{-k},p^{-k}))$, see  Theorem \ref{thm;action_hecke_op_notcoprime}, which is the rationale for this choice.  

For the next theorem we fix some notation:\newline
Let $N(\calM_p)$ be the normalizer of $\calM_p$ in $\calQ_p$. Then the group $W = N(\calM_p)/\calM_p$ is called {\it Weyl group}. It is isomorphic to the symmetric group $S_2$ and acts on $\calM_p$ by changing the entries $t_1,t_2$ of a matrix $m(t_1,t_2)$. 

Let $(k,l)\in \Lambda$. By $\calH^+(\calQ_p//\calK_p, \omega_p)$ we mean the subalgebra of $\calH(\calQ_p//\calK_p, \omega_p)$ generated by $T$ with
\begin{equation}\label{eq:generators_ncopr}
T=
\begin{cases}
  T_{k,l}, & k <l,\\
  T_k, & k=l
  \end{cases}
\end{equation}
as specified in the first line of  \eqref{eq:T_k} and Corollary \ref{cor:struct_hecke_algebra}. In order to state results for all generators of $\calH^+(\calQ_p//\calK_p, \omega_p)$ in subsequent sections we often write $T_{k,k}$ instead of $T_k$.  
Furthermore, let
\begin{equation}\label{eq:generators_copr}
  \begin{split}
  & \tau_k =
  \mathbbm{1}_{\calD_pm(p^k,p^k)\calD_p}\cdot\id_{S_{L_p}},\\
  &  \tau_{k,l} =
  \mathbbm{1}_{\calD_pm(p^k,p^l)\calD_p}\cdot\id_{S_{L_p}} + \mathbbm{1}_{\calD_pm(p^l,p^k)\calD_p}\cdot\id_{S_{L_p}}.\\
  \end{split}
  \end{equation}
Then we denote  by
$\calH(\calM_p//\calD_p, \omega_{p_{|S_{L_p}^{N(\Z_p)}}})^W$  the subalgebra of $\calH(\calM_p//\calD_p, \omega_{p_{|S_{L_p}^{N(\Z_p)}}})$ generated by
  $\tau_{k,l}$ and $\tau_k$, which is nothing else but the subalgebra of all elements of $\calH(\calM_p//\calD_p,\omega_{p_{|S_{L_p}^{N(\Z_p)}}})$ invariant under the Weyl group $W$.

\begin{theorem}\label{thm:satake_isom}
Let $p$ be a prime dividing $|D|$ and $D_p$ anisotropic.
 
  Then the Hecke algebras $\calH^+(\calQ_p//\calK_p, \omega_p)$ and $\calH(\calM_p//\calD_p, \omega_{p_{|S_{L_p}^{N(\Z_p)}}})^W$ are isomorphic. 
  \end{theorem}
\begin{proof}
  In view of Remark \ref{rem:satake_map}, it suffices to prove that $\calS$ is injective and surjective. To this end, we compute $\calS(T)$ for a non-zero $T\in \calH^+(\calQ_p//\calK_p, \omega_p)$. By Corollary \ref{cor:struct_hecke_algebra}, we may assume that $T$ is either $T_{k,l}$ or $T_{k}$ with $(k,l)\in \Lambda_+$.  

We first consider the case $k<l$. Thus,  $T=T_{k,l}\in \calH^+(\calQ_p//\calK_p, \omega_p)$ with $\supp(T_{k,l}) = \calK_pm(p^k,p^l)\calK_p$. Let $m(p^i, p^j)\in \calM_p$ for arbitrary $i,j\in \Z$ with $i \le j$ and $i+j$ a square. One can prove  (see \cite{De}, Lemma 8.24) that $m(p^i,p^j)N(\Q_p)\cap \calK_p m(p^k,p^l)\calK_p\not=\emptyset$ if and only if $i,j\ge k$ and $i+j=k+l$. Therefore,
\begin{equation}\label{eq:sup_satake}
  \supp(\calS (T_{k,l})) \subset \{\calD_pm(p^\nu,p^{k+l-\nu})\calD_p\;|\; \nu=k,\dots, l\}.
\end{equation}


{\it Cartan decompositions and explicit representatives of $\Q_p/ \Z_p$:}

Let $0\not=x=p^rs$ with $s\in \Z_p^\times$. We distinguish two cases:

1. $k+l-2\nu-r\ge 0$:

First note that this inequality is always fulfilled for $\nu=k$ for all $r \le 0$.
Employing the Cartan decomposition produces for all $\nu$ 
\begin{equation}\label{eq:double_coset_decom}
  m(p^\nu,p^{k+l-\nu})n(p^rs)
=  n\_(p^{k+l-2\nu-r}s^{-1})m(p^{\nu+r},p^{k+l-\nu-r})m(s)n(-p^{-r}s^{-1})w.
\end{equation}
It follows that the matrix $m(p^\nu, p^{k+l-\nu})n(p^rs)$ lies in the double coset $\calK_p m(p^k,p^l)\calK_p$ if and only if $r=k-\nu$.
Thus, for $\nu=k+1,\dots, l$ the sum
 $\sum_{x\in \Q_p/\Z_p}T_{k,l}(m(p^\nu,p^{k+l-\nu})n(x))$
runs over all elements of the form $x=p^{k-\nu}s$, $s$ traversing the set
\[
\calU(\nu) = \left\{\sum_{i=0}^{\nu-k-1}x_ip^i\; |\; x_0\in (\Z/p\Z)^\times \text{ and } x_i \in \Z/p\Z,\; i=1,\dots, \nu-k-1\right\}.
\]
For $\nu=k$ this sum consists of a single summand corresponding to $x=0\in \Q_p/\Z_p$.

2. $k+l-2\nu-r< 0$:

We find
\begin{equation}
    m(p^\nu,p^{k+l-\nu})n(p^rs) =  n(p^{2\nu+r-k-l}s)m(p^\nu,p^{k+l-\nu}).
\end{equation}
As the latter inequality is never satisfied for $\nu = k$ for any $r<0$, $m(p^\nu, p^{k+l-\nu})n(p^rs)$ is contained in $\calK_p m(p^k,p^l)\calK_p$ if and only if $\nu = l$. The latter equation can be written as
\begin{equation}\label{eq:double_coset_decomp_2} 
m(p^l,p^k)n(p^rs) =  n(p^{l-k+r}s)w^{-1}m(p^k,p^{l})w. 
\end{equation}
As $k-l-r< 0$ is equivalent to $r > k-l$, the sum
$\sum_{x\in \Q_p/\Z_p}T_{k,l}(m(p^l,p^k)n(x))$ runs over all $x\in \Q_p/\Z_p$ with $|x|_p < l-k$. Assuming a representation of the form $x=p^rs$, we may put $r=k-l$ and write 
\[
\sum_{x\in \Q_p/\Z_p}T_{k,l}(m(p^l,p^k)n(x)) = \sum_{s\in \calU(l)^0}T_{k,l}(m(p^l,p^k)n(p^{k-l}s)),
\]
where
\[
\calU(l)^0 =   \left\{\sum_{i=1}^{l-k-1}x_ip^i\; |\;  x_i \in \Z/p\Z,\; i=1,\dots, l-k-1\right\}.
\]
Note that $|x|_p < 1$ for all $x\in \calU(l)^0$.  

Consequently, $\calU(l)\cup \calU(l)^0$ contains all principal parts $x$ in $\Q_p/\Z_p$ with $\nu_p(x) \ge k-l$. As already pointed out, these are all $x$, for which $T_{k,l}(m(p^l,p^k)n(x))$ is non-zero. 

{\it Computation of $(\calS T_{k,l})(m(p^\nu, p^{k+l-\nu}))$}:

By means of the decompositions \eqref{eq:double_coset_decom} and \eqref{eq:double_coset_decomp_2}, we are now able to compute $(\calS T_{k,l})(m(p^\nu, p^{k+l-\nu}))$ explicitly for any  $\nu\in \{k,\dots,l\}$. Since the computations for $\nu=l$ are more complicated, we treat them separately afterwards. Thus, let $\nu\in \{k,\dots,l-1\}$. Then
\begin{equation}\label{eq:img_satake}
  \begin{split}
  &  \sum_{x\in\Q_p/\Z_p}T_{k,l}\left(m(p^{\nu},p^{k+l-\nu})n(x)\right)_{|S_{L_p}^{N(\Z_p)}}\\
    &=
    \begin{cases}
      \sum_{s\in \calU(\nu)}T_{k,l}\left(m(p^{\nu}, p^{k+l-\nu})n(p^{k-\nu}s)\right)_{|S_{L_p}^{N(\Z_p)}}, & \nu\not=k,\\
      T_{k,l}(m(p^k,p^l))_{|S_{L_p}^{N(\Z_p)}}, & \nu=k
    \end{cases} \\
    &= 
\begin{cases}
  \sum_{s\in \calU(\nu)} \omega_p(n\_(p^{l-\nu}s^{-1}))\circ T_{k,l}(m(p^k,p^l))\circ\omega_p(m(s)n(-p^{\nu-k}s^{-1})w)_{|S_{L_p}^{N(\Z_p)}}, &  \nu\not=k,   \\
  T_{k,l}(m(p^k,p^l))_{|S_{L_p}^{N(\Z_p)}}, &  \nu = k. 
  \end{cases}
  \end{split}
  \end{equation}

Since the level of $L_p$ is $p$, 
the last expression in \eqref{eq:img_satake}  simplifies to
\begin{align*}
\sum_{s\in \calU(\nu)}T_{k,l}(m(p^k, p^l))\circ\omega_p\left(m(s)w\right)_{|S_{L_p}^{N(\Z_p)}}.
\end{align*} 

With the help of the explicit formulas \eqref{eq:weil_rep_explicit} of $\omega_p$  and Lemma \ref{lem:struct_hecke_double_coset_1}, we  obtain
\begin{align*}
  & (\calS T_{k,l})(m(p^\nu, p^{k+l-\nu}))\varphi_p^{(0)} \\
  &= \delta(m(p^\nu,p^{k+l-\nu}))^{1/2}\frac{\gamma_p(D_p)}{|D_p|^{1/2}}\sum_{\gamma\in D_p}\sum_{s\in \calU(\nu)}\leg{s}{|D_p|}T_{k,l}(m(p^k, p^l))\varphi_p^{(s^{-1}\gamma)}\\
  &=
  \begin{cases}
    0, & \text{ if } |D_p|=p,\\
   \delta(m(p^\nu,p^{k+l-\nu}))^{1/2}\gamma_p(D_p)|D_p|^{1/2}|\calU(\nu)|\varphi_p^{(0)}, & \text{ if } |D_p|=p^2.
    \end{cases}\\
\end{align*}
In view of the discussion above, for  $\nu=l$ we have
\begin{equation}\label{eq:T_k_l_m_p_l_k}
  \begin{split}
  &  \sum_{x\in \Q_p/\Z_p}T_{k,l}(m(p^l,p^k)n(x))\varphi_p^{(0)} =\\
    & \sum_{s\in\calU(l)}\omega_p(n\_(s^{-1}))T_{k,l}(m(p^k,p^l))\omega_p(m(s)w)\varphi_p^{(0)} + \sum_{s\in \calU(l)^0}\omega_p(w^{-1})T_{k,l}(m(p^k,p^l))\omega_p(w)\varphi_p^{(0)}.
    \end{split}
  \end{equation}

With the help of Lemma \ref{lem:local_weil_lower_triangular} and the calculations before, it can be verified that the first summand of the above expression is equal to 
\begin{equation}\label{eq:T_k_l_r_k_l}
  \begin{split}
  &  \leg{-1}{|D_p|}\gamma_p(D_p)^2\sum_{s\in \calU(l)}\sum_{\nu_p\in D_p}\psi_p(sq(\nu_p))\varphi_p^{(\nu_p)} \\
    &= \sum_{\nu_p\in D_p}\left(|\calU(l)^0|\sum_{x_0\in (\Z/p\Z)^\times}e(x_0q(\nu_p))\right)\varphi_p^{(\nu_p)},
    \end{split}
  \end{equation}
where we exploited for the last equation the fact that level of $L_p$ is $p$ and that $\gamma_p(D_p)^2 = \leg{-1}{|D_p|}$ (see e. g. \cite{Ze}, p. 73).
Similarly, the second summand can be evaluated to be
\begin{equation}\label{eq:T_k_l_r_smaller}
 |\calU(l)^0|\sum_{\nu_p\in D_p}\varphi_p^{(\nu_p)}.
\end{equation}
Replacing the right-hand side of \eqref{eq:T_k_l_m_p_l_k} with \eqref{eq:T_k_l_r_k_l} and \eqref{eq:T_k_l_r_smaller}, yields
\begin{align*}
  &  \sum_{x\in \Q_p/\Z_p}T_{k,l}(m(p^l,p^k)n(x))\varphi_p^{(0)}\\
  &=|\calU(l)^0|\sum_{\nu_p\in D_p}\left(\sum_{x\in \Z/p\Z}e(xq(\nu_p))\right)\varphi_p^{(\nu_p)}\\
  &= |\calU(l)^0|p\varphi_p^{(0)}.
  \end{align*}
Here we have used the standard formula for the exponential sum $\sum_{x\in \Z/p\Z}e(xq(\nu_p))$. 
  This leads us finally to
  \begin{equation}\label{eq:image_satake_map}
    \begin{split}
&  (\calS T_{k,l})(m(t_1,t_2)) =\delta(m(t_1,t_2))^{1/2}\times \\
&  \begin{cases}
    \mathbbm{1}_{\calD_pm(p^k,p^l)\calD_p}\id_{S_{L_p}^{N(\Z_p)}}, & m(t_1, t_2) = m(p^k, p^l),\\
    \delta_p\gamma_p(D_p)|D_p|^{1/2}|\calU(\nu)|\mathbbm{1}_{\calD_pm(p^\nu,p^{k+l-\nu})\calD_p}\id_{S_{L_p}^{N(\Z_p)}}, & m(t_1,t_2) = m(p^\nu, p^{k+l-\nu}), \nu\not=k,l \\
     p^{l-k}\mathbbm{1}_{\calD_pm(p^l,p^{k})\calD_p}\id_{S_{L_p}^{N(\Z_p)}}, & m(t_1,t_2) = m(p^l,p^k),\\
    {\bf 0}, & \text{ otherwise,}
   \end{cases} \\
      &    = p^{\frac{1}{2}(l-k)}\times \\
      &\begin{cases}
         \mathbbm{1}_{\calD_pm(p^k,p^l)\calD_p}\id_{S_{L_p}^{N(\Z_p)}}, & m(t_1, t_2) = m(p^k, p^l),\\
         \delta_p(p-1)\gamma_p(D_p)\mathbbm{1}_{\calD_pm(p^\nu,p^{k+l-\nu})\calD_p}\id_{S_{L_p}^{N(\Z_p)}}, & m(t_1,t_2) = m(p^\nu, p^{k+l-\nu}), \nu\not=k,l \\
         \mathbbm{1}_{\calD_pm(p^l,p^{k})\calD_p}\id_{S_{L_p}^{N(\Z_p)}}, & m(t_1,t_2) = m(p^l,p^k),\\
    {\bf 0}, & \text{ otherwise,}
         \end{cases}
   \end{split}   
\end{equation}
  where $\delta_p=1$ if $|D_p|=p^2$ and zero otherwise. 

If $k=l$, it follows from \eqref{eq:sup_satake} that $\supp(\calS T_k) =\calD_pm(p^k,p^{k})\calD_p$. The same thoughts as for $k< l$ after equation \eqref{eq:double_coset_decom} yield
\begin{align*}
  (\calS T_{k}(m(p^k,p^k)) &= T_{k}(m(p^k,p^k))_{|S_{L_p}^{N(\Z_p)}}\\
  &=\mathbbm{1}_{\calD_pm(p^k,p^k)\calD_p}\id_{S_{L_p}^{N(\Z_p)}}.
\end{align*}
From the above  follows immediately that $\calS$ is injective. For the surjectivity it suffices to proof that $\tau_k$ and $\tau_{k,l}$ are contained in the image of $\calS$. This can be done almost verbatim as in \cite{De}, p. 212. 
\end{proof}

\begin{remark}\label{rem:polynomial_algebra}
  Combining Theorem \ref{thm:satake_isom}, Theorem \ref{rem:hecke_algebra_mp} and the isomorphism
\begin{equation}\label{eq:hecke_algebra_group_algebra}
\mathbbm{1}_{\calD_pm(p^k,p^l)\calD_p}\mapsto \mathbbm{1}_{m(p^k,p^l)\calD_p},
\end{equation}
which maps $\calH(\calM_p//\calD_p)$ to the group algebra $\C[\calM_p/\calD_p]$,
  shows that $\calH^+(\calQ_p//\calK_p,\omega_p)$ is isomorphic to $\C[\calM_p/\calD_p]^W$. This recovers essentially the classical result of Satake (see \cite{Sa}). Also by Satake, this group algebra can be identified with a polynomial algebra.   
\end{remark}

 Whenever $p$ divides $|D|$ and we deal with elements $T_{k,l}$, $T_k$ or $\tau_{k,l}$, $\tau_k$  of either of the Hecke algebras $\calH^+(\calQ_p//\calK_p,\omega_p)$ or $\calH(\calM_p//\calD_p, \omega_{p_{|S_{L_p}^{N(\Z_p)}}})^W$, we mean the above stated and assume that $D_p$ is anisotropic.

\subsection{The case of primes $p$ not dividing $|D|$}
We denote with $\calH(\calM_p//\calD_p)^W$  the subalgebra of all elements of $\calH(\calM_p//\calD_p)$ invariant under the Weyl group $W$. In this case it easily seen that $\calH(\calQ_p//\calK_p, \omega_p)$ is isomorphic to $\calH(\calM_p//\calD_p)^W$.

\begin{theorem}\label{thm:satake_isom_coprime}
  Let $p$ be a prime coprime to $|D|$. Then the Hecke algebras $\calH(\calQ_p//\calK_p, \omega_p)$ and $\calH(\calM_p//\calD_p)^W$ are isomorphic as algebras. 
  \end{theorem}
\begin{proof}
  By Lemma 3.4 in \cite{St}, we know that $L_p$ is unimodular and $\omega_p$ is the trivial representation on the space $S_{L_p} = \C\varphi_p^{(0)}$. A basis of $\calH(\calQ_p//\calK_p, \omega_p)$ is then given by
\[
\left\{\mathbbm{1}_{\calK_pm(p^k, p^l)\calK_p}\cdot\id_{S_{L_p}}\; |\; (k,l)\in \Lambda_+\right\}.
\]
The composition of
\[
F: \calH(\calQ_p//\calK_p,\omega_p)\rightarrow \calH(\calQ_p//\calK_p),\quad \mathbbm{1}_{\calK_pm(p^k, p^l)\calK_p}\cdot\id_{S_{L_p}}\mapsto \mathbbm{1}_{\calK_pm(p^k, p^l)\calK_p}
\]
with the classical Satake map (see e. g. \cite{De} or \cite{Ca}) gives the desired isomorphism.
  \end{proof}

\section{Vector valued automorphic forms and vector valued modular forms}\label{sec:automorphic_forms}
In his thesis \cite{We}, Werner assigned to each vector valued modular form a vector valued automorphic form on $\GL_2(\A)$. He also provided an adelic Hecke operator, which corresponds to the Hecke operator $T\kzxz{p}{0}{0}{1}$ on a space of vector valued modular forms.
In this section we continue this work and embed it into a  more general framework of vector valued automorphic forms. More specifically, we
\begin{enumerate}
\item
  describe the image $A_\kappa(\omega_f)$ of $S_\kappa(\rho_L)$ under Werner's adelization map and establish that it is an Hilbert space isomorphism (see Theorem \ref{thm:correspondence_mod-adelic}).
\item
  define an action of the whole Hecke algebra $\calH(\calQ_p//\calK_p,\omega_p)$ on $A_\kappa(\omega_f)$. 
  \end{enumerate}
However, we stick to the Hecke operators given in \cite{BS} instead of using Werner's Hecke operators. As a consequence, we have to work with the extension of the Weil representation as given in \cite{St2}, Section 4,  Section \ref{sec:preliminaries} and its adelic counterpart in Section \ref{subsec:weil_repr_adelic} of the present paper.

Instead of working with $\GL_2(\A)$, we consider the restricted product
\begin{equation}\label{def:modified_adeles}
\calG(\A) =\sideset{}{'} \prod_{p\le \infty}\calQ_p = \left\{(g_p)\in \prod_{p\le \infty}\calQ_p\; |\; g_p\in \calK_p \text{ for almost all primes } p\right\},
\end{equation}
where
\[
\calQ_\infty = \{M\in \GL_2(\R)\; |\; \det(M) \in (\R^\times)^2\}. 
\]
Note that $\calK_\infty = \SO(2)$ is a subgroup of $\calQ_\infty$. The group
$\calG(\Q)$ 
can be embedded diagonally  as a discrete subgroup of $\calG(\A)$.
An important decomposition for $\GL_2(\A)$, which will be needed for the definition of automorphic forms, is the strong approximation. An analogous result holds for $\calG(\A)$.
\begin{theorem}\label{thm:strong_approximation}
  Let $\calK = \prod_{p<\infty}\calK_p\subset \GL_2(\A_f)$. Then
  \begin{equation}\label{eq:strong_approx}
    \calG(\A) = \calG(\Q)(\calQ_\infty \times \calK). 
  \end{equation}
  More generally, let $\calU = \prod_{p<\infty}\calU_p$ be any open compact subgroup of $\calK$ with the property that $\det(\calU) = (\widehat{\Z}^\times)^2$. Then
  \begin{equation}\label{eq:strong_approx_subgroups}
    \begin{split}
      &\calG(\A_f) = \calG(\Q)\cdot \calU \text{ and }\\
      &\calG(\A) = \calG(\Q)(\calQ_\infty \times \calU).
      \end{split}
    \end{equation}
\end{theorem}
\begin{proof}
  A proof for the classical result for $\GL_2(\A)$ can be found in many places, among them in \cite{KL}, Section 5.2 and Section 6.3. One can check that the proofs of Proposition 5.10, Proposition 6.5 and Theorem 6.8 of \cite{KL} carry over to the analogous statements in our setting.  
\end{proof}

The Weil representation $\omega_f$ is extended in Section \ref{subsec:weil_repr_adelic} to the group $\calK'$. The projection $\calK'\rightarrow \calK,\quad (k,r)\mapsto k$ is surjective. Thereby, in view of Theorem \ref{thm:strong_approximation}, for every $g\in \calG(\A)$ the decomposition $g = \gamma(g_\infty \times k), \; \gamma\in \calG(\Q), g_\infty \in \Q_\infty$ and $k\in \calK$ can be lifted to $(g,r)$ by $\gamma(g_\infty \times (k, r))$ with $\det((\Pi_N\circ \Pi)(k)) = r^2$.

In \cite{KL} and \cite{Ge} functions $f:\GL_2(\Q)\setminus \GL_2(\A)\rightarrow \C$ with certain properties were related to (scalar valued) elliptic modular forms. To incorporate the square root factor for the Weil representation $\omega_f$,  we consider $\calG(\Q)$-invariant   $S_L$-valued functions on the adelic group
\begin{equation}\label{eq:lifted_adelic_group}
\widetilde{\mathcal G}(\mathbb A)
=
\operatorname{im}\Bigl(
\calG(\mathbb Q)\times \calQ_\infty\times\calK'
\longrightarrow
\calG(\mathbb A)\times(\mathbb Z/N\mathbb Z)^\times
\Bigr),
\end{equation}
where
\[
(\gamma,g_\infty,(k,r))
\longmapsto
\bigl(\gamma(g_\infty\times k),r\bigr).
\]
By abuse of notation, we continue to write $\calG(\A)$ for the corresponding lifted adelic group and omit the square root factor whenever no confusion can arise. Thus, we study functions
\[
F: \calG(\Q)\setminus \calG(\A)\rightarrow S_L
\]
with a similar goal as for the classical adelic automorphic forms.
With respect to the basis $\{\varphi_\mu\}_{\mu\in D}$ of $S_L$ such a function can be written in the form $F = \sum_{\mu\in D}F_\mu\varphi_\mu$. In view of \eqref{eq:local_decomp_S_L} and \eqref{eq:local_S_L},   we will only consider {\it factorizable  functions}, that is, only those $S_L$-valued functions $F$ which possess a decomposition of the form 
\[
F(\gamma(g_\infty\times g_f)) = \bigotimes_{p<\infty} F_p(g_\infty,g_p),
\]
where
\[
F_p(g_\infty, g_p) =
\begin{cases}
  \sum_{\lambda\in D_p}F_{\lambda,\infty}(g_\infty)F_{\lambda,p}(g_p)\varphi_p^{(\lambda)}, & p\mid |D|,\\
  \varphi_p^{(0)}, & p\nmid |D|.
  \end{cases}
\]
Using the bilinearity of the tensor product, we have
\[
F_\mu(\gamma(g_\infty\times g_f)) = \langle \sum_{(\lambda_p)_{p} \in \bigoplus_{p<\infty} D_p}\prod_{p<\infty} F_{\lambda_p,\infty}(g_\infty)F_{\lambda_p,p}(g_p)\bigotimes_{p<\infty}\varphi_p^{(\lambda_p)}, \varphi_\mu\rangle.
\]
Note that $F$ is well defined since any occurring sum, product or tensor product is finite.

We denote the space of all these functions $F: \calG(\Q)\setminus \calG(\A)\rightarrow S_L$ with $\calF_L$.
Associated to the Weil representation $\omega_f$ on the space $S_L$ we define
\begin{equation}\label{der:L_2_spaces}
  \begin{split}
 &   \calL^2(\calG(\Q)\setminus \calG(\A), \omega_f) = \left\{F\in \calF_L\; \left|
    \begin{array}{ll}
      \text{i)} &  F_\mu \text{ is measurable } \text{ for all } \mu \in D\\
      \text{ii)} & F(zg) = \omega_f(z_f)^{-1}F(g) \text{ for all }\\
      & z=z_\Q(z_\infty\times z_f)\in \calZ(\A)\\
      \text{iii)} & \int_{\overline{\calG}(\Q)\setminus \overline{\calG}(\A)}\| F(g)\|^2 dg <\infty
      \end{array}
    \right.\right\}\\
   & \text{ and }\\
    & \calL^2_0(\omega_f) \\
    &= \left\{F\in \calL^2(\calG(\Q)\setminus \calG(\A), \omega_f)\; \left|\; \int_{N(\Q)\setminus N(\A)}F_\mu(ng)dn = 0 \text{ for all } \mu \in D,   \text{  a. e. } g\in \calG(\A)\right.\right\}.
\end{split}
\end{equation}
Here by
\begin{enumerate}
  \item[i)]
    $\|F(g)\|^2$ we mean $\langle F(g), F(g)\rangle$ as defined in \eqref{def:L_2_S_L},
\item[ii)]
  $\overline{\calG}(R) = \calZ(R)\setminus \calG(R)$, where $\calZ(R)$ is the center of $\calG(R)$ and $R$ stands for any commutative ring with $1$,
\item[iii)]
  $dg$ and $dn$ we mean the Haar measure on $\overline{\calG}(\Q)\setminus \overline{\calG}(\A)$ and $N(\Q)\setminus N(\A)$, respectively.
  \end{enumerate}

Measurability for each component function $F_\mu$  is meant in the sense of Proposition 7.15 of \cite{KL}: $F_{\mu}$ can be written as a product $\prod_{p< \infty}F_{\mu,p}(g_p)$, each component satisfying:
\begin{enumerate}
\item[i)]
  $F_{\mu,p}:\calQ_p\rightarrow \C$ is measurable for all $p\le \infty$
\item[ii)]
$F_{{\mu,p}_{|_{\calK_p}}} = 1$  for all $p\notin S,$ where $S$ is a finite set of places. 
\end{enumerate}

The above integrals over $\overline{\calG}(\Q)\setminus \overline{\calG}(\A)$ and $N(\Q)\setminus N(\A)$  are explained in \cite{KL}, Proposition 7.43 and Proposition 12.2, and  meant in the very same way.
Also note that the integral in iii) of $\calL^2(\calG(\Q)\setminus \calG(\A), \omega_f)$ is well defined as $F$ satisfies ii) and the Weil representation $\omega_f$ is unitary with respect to $\langle \cdot, \cdot\rangle$.
The spaces $\calL^2(\calG(\Q)\bs \calG(\A),\omega_f)$ and $\calL^2_0(\omega_f)$ are subspaces of the spaces $L^2(\calG(\Q)\bs \calG(\A),\omega_f)$ and $L^2_0(\omega)$, respectively, which are defined the same way but without the assumption that the functions are factorizable.

Werner assigned in \cite{We}, Def. 49,  a $\C[D]$-valued Function $F_f$ on $\calG(\Q)\setminus\calG(\A)$ to a cusp form $f\in S_\kappa(\rho_L)$. We adopt his definition to our setting, which basically means  that we replace the group ring with the isomorphic space $S_L$.

\begin{definition}\label{def:mod_adelic}
  Let $f\in S_\kappa(\rho_L)$ and $g\in \calG(\A)$ with $(g,r)=\gamma(g_\infty \times (k,r))$, where  $\gamma\in \calG(\Q), g_\infty\in \calQ_\infty$ and $k\in \calK$. Then in terms of this decomposition we define a map $\mathscr{A}$
\begin{equation}\label{eq:mod_adelic}
f\mapsto \mathscr{A}(f)=F_f\;\text{ with } F_f(g)= \omega_f(k,r)^{-1}j(g_\infty,i)^{-\kappa}f(g_\infty i).
\end{equation}
\end{definition}

Lemma 50 in \cite{We} shows that the definition of $F_f$ in \eqref{eq:mod_adelic} is independent of the decomposition of $g$. Moreover, from its definition it follows immediately that $F_f$ is $\calG(\Q)$-invariant.  

\begin{proposition}\label{rem:correspondence_L2}
  Let $f\in S_\kappa(\rho_L)$. Then the assigned function $F_f$ on $\calG(\Q)\setminus\calG(\A)$ lies in the space $\calL^2(\calG(\Q)\setminus\calG(\A), \omega_f)$. 
\end{proposition}
\begin{proof}
  \begin{enumerate}
  \item[i)]
    By definition, the $\mu$-th component of $F_f$ is given by
    \begin{equation}\label{eq:component_automorphic_form}
      \begin{split}
      (F_f)_\mu(g)&= \langle \omega_f(k)^{-1}j(g_\infty,i)^{-\kappa}f(g_\infty i), \varphi^{(\mu)}\rangle\\
      &=j(g_\infty, i)^{-\kappa}\sum_{\lambda\in D}f_\lambda(g_\infty i)\prod_{p<\infty} \langle\omega_p(k_p)^{-1}\varphi_p^{(\lambda_p)},\varphi_p^{(\mu_p)}\rangle,
        \end{split}
      \end{equation}
    where $g=\gamma(g_\infty \times k)$.  It is well known  that  $j(g_\infty, i)^{-\kappa}f_\lambda(g_\infty i)$ is  measurable on $\calQ_\infty$ as $f_\lambda$ is a scalar valued cusp form for $\Gamma(N)$ (cf. \cite{Ge}, $\S 2$, for this case). 
    As a result of the discussion in Chapter \ref{subsec:weil_repr_adelic}, we have that $\omega_p$ is trivial for all $p\nmid N$. For $p\mid N$ we find by means of the explicit formulas of $\omega_p$ (see \eqref{eq:weil_rep_explicit} or  \cite{BY}, p. 645,  or \cite{St}, Lemma 3.4)   that $\omega_p$ is trivial on  the subgroup
    \[
    \calK_p(p^{\ord_p(D)}) = \left\{\kzxz{a}{b}{c}{d}\in \calK_p\; |\; \kzxz{a}{b}{c}{d}\equiv \kzxz{1}{0}{0}{1}\bmod{p^{\ord_p(D)}}\Z_p\right\}
      \]
      and factors thereby through  $\calK_p/\calK_p(p^{\ord_p(D)})$ for each $p$ dividing $N$.
      Since \newline $\calK_p(p^{\ord_p(D)})$ has as compact subgroup a finite measure,
       $\langle\omega_p(k_p)^{-1}\varphi_p^{\mu_p},\varphi_p^{\mu_p}\rangle$ is a measurable function for all primes $p$. We then obtain that $(F_f)_\mu$ is measurable in the above stated sense. 
  \item[ii)]
    Let $z = z_{\Q}(z_\infty\times z_f)\in \calZ(\A)$. Then it follows immediately from the definition of $F_f$ that $F_f(zg) = \omega_f(z_f)^{-1}F_f(g)$.
  \item[iii)]
    It can be verified that Proposition 7.43 and the discussion before of \cite{KL} is also valid in our situation. We have to check that all steps of the proof are still working if we replace the involved groups by the corresponding groups in our setting. This is  in fact the case, some steps are even easier since we only have to deal with  matrices whose determinant is a square. As a result, we may replace the integral over $\overline{\calG}(\Q)\setminus \overline{\calG}(\A)$ with the corresponding integral over $D\calK_\infty \times \calK$. Here $D$ is a fundamental domain for $\Gamma\setminus \H$ interpreted as subset of $\SL_2(\R)$ (not to be confused with the discriminant group). Following the proof of Proposition 12.15 in \cite{KL},  we find for $F_f$
    \begin{equation}\label{eq:adelic_integrals_calc}
      \begin{split}
      \int_{\overline{\calG}(\Q)\setminus \overline{\calG}(\A)}\| F_f(g)\|^2 dg &= \int_{D\calK_\infty}\int_{\calK}\|F_f(g\times k)\|^2dkdg\\
      &=\int_{D}\|j(g_\infty,i)^{-\kappa}f(g_\infty i)\|^2dg,
      \end{split}
      \end{equation}
    where we have used that $\omega_f$ is unitary with respect to $\langle \cdot, \cdot\rangle$ and that the Haar measure on $\calQ_p$ is normalized to be equal to one on $\calK_p$ for $p\le \infty$. If we identify $g_\infty i$ with an element $\tau\in \Gamma\setminus \H$, the last integral in \eqref{eq:adelic_integrals_calc} becomes
    \[
    \int_{\Gamma\setminus \H}\|f(\tau)\|^2\im(\tau)^\kappa\frac{dxdy}{y^2},
    \]
    which is the Petersson norm of $f\in S_\kappa(\rho_L)$ and therefore $<\infty$. Thus, the $L^2$-norm of $F_f$ is finite. 
  \end{enumerate}
\end{proof}

\begin{lemma}\label{lem:cuspidal}
  Let $f\in S_\kappa(\rho_L)$ and $F_f$ the assigned automorphic form given by \eqref{eq:mod_adelic}. Then
  \[
  \int_{N(\Q)\setminus N(\A)}F_\mu(ng)dn = 0
  \]
  for almost every $g\in \calG(\A)$ and all $\mu\in D$. 
\end{lemma}
\begin{proof}
  The proof proceeds along the lines of the one of Proposition 12.2 in \cite{KL}.
  Let $n= n(x_\Q)(n(x_\infty)\times n(x_f))\in N(\A)$ and $g=\gamma(g_\infty \times g_f)\in \calG(\A)$. Then the definition of $F_f$ and $\omega_f$  yields
  \begin{align*}
    F_\mu(ng) &= \langle j(g_\infty, i)^{-\kappa}j(n(x_\infty),g_\infty i)^{-\kappa}\omega_f(g_f)^{-1}\omega_f(n(x_f))^{-1}f(n(x_\infty)(g_\infty i)), \varphi_\mu\rangle\\
    &= j(g_\infty, i)^{-\kappa}j(n(x_\infty),g_\infty i)^{-\kappa}\sum_{\nu\in D}\psi_f(-x_fq(\nu))f_\nu(n(x_\infty)(g_\infty i))\langle \omega_f(g_f)^{-1}\varphi_\nu,\varphi_\mu\rangle.
  \end{align*}
  As suggested in \cite{KL}, Prop. 12.2., we calculate  more generally for $r\in \Q$
  \begin{equation}\label{eq:fourier_integral}
    \begin{split}
    &    \int_{N(\Q)\setminus N(\A)}F_\mu(n(x)g)\psi(rx)dx \\
    & = j(g_\infty,i)^{-\kappa}\sum_{\nu\in D}\langle \omega_f(g_f)^{-1}\varphi_\nu,\varphi_\mu\rangle\times \\
   & \int_{\N(\Z)\setminus (N(\R)\times N(\widehat{\Z}))}\psi_f(-x_fq(\nu))f_\nu(n(x_\infty)(g_\infty i))\psi_\infty(rx_\infty)\psi_f(rx_f)dx_fdx_\infty.
      \end{split}
    \end{equation}
  We can write the integral in the last expression as
  \begin{align*}
    \int_0^1f_\nu(n(x_\infty)(g_\infty i))\psi_\infty(rx_\infty)\int_{N(\widehat{\Z})}\psi_f((r-q(\nu))x_f)dx_f dx_\infty,
  \end{align*}
  where the integral over $N(\widehat{\Z})$ is one if and only if $r\in \Z+q(\nu)$.
  For such $r$ (note that $\psi_\infty(x_\infty) = e(-x_\infty)$), taking into account that $\int_0^1f_\nu(x_\infty+\tau)e(-rx_\infty)dx_\infty = e(r\re(\tau))c(\nu,r)$, where $c(\nu,r)$ is the Fourier coefficient of $f$ with respect to $(\nu,r)$, we finally obtain
  \begin{align*}
    \int_{N(\Q)\setminus N(\A)}F_\mu(n(x)g)\psi(rx)dx = j(g_\infty,i)^{-\kappa}\sum_{\nu\in D}\langle \omega_f(g_f)^{-1}\varphi_\nu,\varphi_\mu\rangle e(r\re(\tau))c(\nu, r),
  \end{align*}
  where $\tau = g_\infty i$. Since $f$ is a cusp form, we have that for $r=0$ all coefficients $c(\nu, r)$ vanish. This  gives the desired result.  
\end{proof}

The next theorem characterizes the image of $S_\kappa(\rho_L)$ under the map $\mathscr{A}$ in \eqref{eq:mod_adelic} more closely. 

\begin{theorem}\label{thm:correspondence_mod-adelic}
  Let $A_{\kappa}(\omega_f)$ be the space of functions $F\in \calL^2_0(\omega_f)$ satisfying
  \begin{enumerate}
  \item[i)]
    $    F(gk) = \omega_f(k)^{-1}F(g)$ for all $k\in \calK$ and all $g\in\calG(\A)$
  \item[ii)]
    $F(g\kzxz{\cos(\theta)}{\sin(\theta)}{-\sin(\theta)}{\cos(\theta)}) = e^{i\kappa\theta}F(g)$ for all $\theta\in [0,2\pi)$ and all $g\in \calG(\A)$
  \item[iii)]
    All the components $F_\mu$ of  $F$, considered as a function of $\calQ_\infty$ alone, satisfy the differential equation $L F_\mu = 0$.  Here $L$ is the differential operator given by
    \begin{equation}
    L = e^{-2i\theta}\left(-2iy\frac{\partial}{\partial x}+2y\frac{\partial}{\partial y} +i \frac{\partial}{\partial \theta}\right)
    \end{equation}
    with respect to the coordinates referring to the decomposition
\begin{equation}\label{eq:decomp_q_infty}
g_\infty = z_\infty \zxz{1}{x}{0}{1}\kzxz{y^{\frac{1}{2}}}{0}{0}{y^{-\frac{1}{2}}}\kzxz{\cos(\theta)}{\sin(\theta)}{-\sin(\theta)}{\cos(\theta)}
\end{equation}
of $g_\infty\in \calQ_\infty$. 
  \end{enumerate}
  Then the map $\mathscr{A}$ defines an isometry from $S_\kappa(\rho_L)$ onto $A_{\kappa}(\omega_f)$. 
\end{theorem}
\begin{proof}
  This theorem is well known for scalar valued automorphic forms, see e. g. \cite{Ge} or \cite{KL}. Most parts of its proof can be settled with reference to the proof of its scalar valued analogue.
  
 Let $f\in S_\kappa(\rho_L)$.  It follows from Proposition \ref{rem:correspondence_L2} and Lemma \ref{lem:cuspidal} that $F_f\in \calL^2_0(\omega_f)$.
 The assertion in i)  is proved in \cite{We}, Theorem 51, the one in ii) results from a straightforward calculation analogous to the scalar valued case (see \cite{KL}, Proposition 12.5).
 For iii) note that $F_\mu(g_\infty \times 1_f) = y^{k/2}e^{ik\theta}f_\mu(x+iy)$ if we decompose $g_\infty\in \calQ_\infty$ according to \eqref{eq:decomp_q_infty}. The same proof as in \cite{KL}, applied to each component $F_\mu$, establishes the result using the assumption $f\in S_\kappa(\rho_L)$.

Kudla \cite{Ku} defined a map that assigns to a vector valued function $F$ on $\calG(\A)$ a vector valued function $f_F$ on $\H$: 
\begin{equation}\label{eq:adelic_mod}
  F\mapsto f_F, \; f_F(\tau) = j(g_\tau,i)^\kappa F(g_\tau\times 1_f),
\end{equation}
where $g_\tau = \zxz{1}{x}{0}{1}\kzxz{y^{\frac{1}{2}}}{0}{0}{y^{-\frac{1}{2}}}$ and $\tau = g_\tau i = x+iy\in \H$. It is easily seen that this map is well-defined and that it is the inverse map of $\mathscr{A}$ (see \cite{KL}, Prop. 12.5, for the corresponding scalar valued result). It remains to show that $f_F$ is an element of $S_\kappa(\rho_L)$ for any $F\in A_{\kappa}(\omega_f)$.
Kudla proved that  $f_F$ transforms like a vector valued modular form with respect to $\omega_f$  if $F\in A_{\kappa}(\omega_f)$ (\cite{Ku}, Lemma 1.1).  Since each component of $F$ satisfies the differential equation in iii), it follows that each component of $f_F$ is holomorphic on the upper half plane (see \cite{KL}, Prop. 12.5).  In view of these two properties, $f_F$ possess a Fourier expansion, see \cite{Br1}, p. 18. By Proposition \ref{rem:correspondence_L2} we know that the Petersson norm of $f_F$ coincides with the $L^2$ norm of $F$, it is in particular finite.
One can prove in the same way as in Prop. 3.39 of \cite{KL} that $f_F$ is an element of $S_\kappa(\rho_L)$. Thus, the map  in \eqref{eq:mod_adelic} is surjective and an isometry. 
\end{proof}

\subsection{The action of $\calH(\calQ_p//\calK_p,\omega_p)$ on $A_{\kappa}(\omega_f)$}
The goal of this subsection is to define an action of $\calG(\A)$ via the Hecke algebra $\calH^+(\calQ_p//\calK_p,\omega_p)$ (or $\calH(\calQ_p//\calK_p,\omega_p)$ if $(p, |D|)=1$) on the space $A_{\kappa}(\omega_f)$ of vector valued automorphic forms. Whenever we write $\calH^+(\calQ_p//\calK_p,\omega_p)$ we tacitly also mean $\calH(\calQ_p//\calK_p,\omega_p)$ in the case $(p,|D|)=1$ and don't mention the latter in the following. Since $\calH^+(\calQ_p//\calK_p,\omega_p)$ acts only on the $p$-component of an element $F\in A_{\kappa}(\omega_f)$, we need to complement the contribution of $\calH^+(\calQ_p//\calK_p,\omega_p)$ with suitable operators for the other places.  The envisaged action will be defined in such a way that it is compatible with the action of Hecke operators on $S_\kappa(\rho_L)$. Werner proposed in \cite{We}, Chapter 6, the definition of an adelic vector valued Hecke operator mimicking Gelbart's approach of an adelic scalar valued Hecke operator. Our approach is more conceptual and transfers the action of the classical spherical Hecke algebra (as for instance in \cite{BP} or \cite{Mu1}, $\S$ 6), to the vector valued setting.   

\begin{definition}\label{deinition:adelic_hecke_op}
  Let $p\in \Z$ be a fixed prime, $g=\gamma(g_\infty\times g_f)\in \calG(\A)$ and $T_p\in \calH^+(\calQ_p//\calK_p,\omega_p)$. Then we define for a fixed $h\in \calG(\A)$
  \begin{equation}\label{eq:shift_op}
    R^{T_p}(h): \calF_L\rightarrow \calF_L, \quad F\mapsto R^{T_p}(h)F= \bigotimes_{q<\infty}R_q^{T_p}(h_q)F_q 
  \end{equation}
  with
  \begin{equation}\label{eq:local_hecke_op}
    R_q^{T_p}(h_q)F_q(g_q) = 
    \begin{cases}
      F_q(g_\infty, g_qh_q), & q\not=p\\
      T_p(h_p)(F_p(g_\infty,g_p)), & q= p.
\end{cases}
  \end{equation}
  The operator
  \begin{equation}\label{eq:global_adelic_hecke_op}
\calT^{T_p}:A_{\kappa}(\omega_f)\rightarrow A_{\kappa}(\omega_f),\quad \calT^{T_p}(F)(g)=\sum_{x_p\in \calQ_p/\calK_p}R^{T_p}(\iota_p(x_p))F(g\iota_p(x_p))
  \end{equation}
  can be interpreted as a  vector valued analogue of the construction in \cite{BP} or \cite{Mu1}.  
  For the sake of better readability, we omit the argument $g_\infty$ in the subsequent calculations and assume tacitly that the local functions also depend on $g_\infty$.
\end{definition}

\begin{remark}\label{rem:prop_adelic_hecke_op}
  \begin{enumerate}
  \item[i)]
    If we decompose $\calT^{T_p}$ into  into its components, we obtain
    \begin{equation}\label{eq:hecke_op_adelich_components}
      \calT^{T_p}(F)(g) = \bigotimes_{q\not=p}F_q(g_g)\otimes \left(\sum_{x_p\in \calQ_p/\calK_p}T_p(x_p)F_p(g_px_p)\right).
      \end{equation}
  Since $T_p\in \calH^+(\calQ_p//\calK_p, \omega_p)$ has compact support, the sum in \eqref{eq:hecke_op_adelich_components} is finite. It can be verified by means of Theorem \ref{thm:correspondence_mod-adelic}, i), and Definition \ref{def:spherical_hecke_algebra}, ii), that \eqref{eq:hecke_op_adelich_components} and therefore  \eqref{eq:global_adelic_hecke_op} is independent of the representative $x_p\in \calQ_p/\calK_p$ and is thus well-defined. We will show later in the paper that $\calT^{T_{p}}(F)$ is indeed contained in $A_{\kappa}(\omega_f)$. 
\item[ii)]  
  Let $p$ be a prime, $T_p, T_p' \in \calH^+(\calQ_p//\calK_p,\omega_p)$.  
  Then  by a straightforward calculation, using \eqref{eq:hecke_op_adelich_components} and the bilinearity of the tensor product, we obtain
  \begin{equation}\label{eq:additive}
    \calT^{xT_p+ yT_p'}(F)(g) = x\calT^{T_p}(F)(g) + y\calT^{T_p'}(F)(g)
  \end{equation}
  for all $F\in A_{\kappa}(\omega_f)$, all $g\in \calG(\A)$ and all $x,y\in \C$.
  
  There is also a compatibility relation regarding convolution: 
  \begin{align*}
  \calT^{T_p\ast T_p'}(F)(g) = \bigotimes_{q\not=p}F_q(g_q)\otimes \left(\sum_{x_p\in \calQ_p/\calK_p}\left(\sum_{y_p\in \calQ_p/\calK_p}T_p(y_p)\circ T_p'(y_p^{-1}x_p)\right)(F_p(g_px_p)) \right).
  \end{align*}
  Since both sums over $\calQ_p/\calK_p$ are finite, we can change their order and obtain
  \begin{align*}
    &  \bigotimes_{q\not=p} F_q(g_q)\otimes \left(\sum_{y_p\in \calQ_p/\calK_p}\left(\sum_{x_p\in \calQ_p/\calK_p}T_p(y_p)\circ T_p'(y_p^{-1}x_p)(F_p(g_px_p))\right)\right) \\
    &= \bigotimes_{q\not=p} F_q(g_q)\otimes\left(\sum_{y_p\in \calQ_p/\calK_p}T_p(y_p) \left(\sum_{z_p\in \calQ_p/\calK_p}T_p'(z_p)(F_p(g_py_pz_p))\right)\right)\\
   &= (\calT^{T_p}\circ \calT^{T_p'})(F)(g).
  \end{align*}
For the second equation, we replaced for each $y_p$ the sum over $x_p$ by a sum over a new variable $z_p$ by means of  the substitution $x_p = y_pz_p$. 
  \end{enumerate}
%
%
\item[iii)]
 Relation to classical adelic Hecke operators:  If the lattice $L$ is unimodular, the discriminant group $D$ is trivial, which implies that the finite Weil representation $\rho_L$ is trivial. Consequently, following Section \ref{subsec:weil_repr_adelic}, the local Weil representation $\omega_p$ is trivial for each prime $p$. This in turn implies (cf. Theorem \ref{thm:satake_isom_coprime}) that $\calH^+(\calQ_p//\calK_p,\omega_p)$ is isomorphic to the classical spherical Hecke algebra $\calH(\calQ_p//\calK_p)$.
 Let $F_f\in A_\kappa(\omega_f)$ with
 \begin{align*}
   F_f(\gamma(g_\infty\times k)) &= j(g_\infty,i)^{-\kappa}f_0(g_\infty i)\omega_f(k)^{-1}\varphi_0 \\
   &= j(g_\infty,i)^{-\kappa}f_0(g_\infty i)\varphi_0
   \end{align*}
 and $T_p = t_p\cdot\id_{S_{L_p}}$ with $t_p = \mathbbm{1}_{\calK_pm(p^k,p^l)\calK_p}$.  Then by Definition \ref{deinition:adelic_hecke_op},
 \begin{align*}
   \calT^{T_p}(F_f)(\gamma(g_\infty\times k)) &= \bigotimes_{q\not=p}\varphi_q^{(0)}\otimes \left(\sum_{x_p\in\calQ_p/\calK_p}T_p(x_p)(F_f)_p(k_px_p)\right)\\
   &= \bigotimes_{q\not=p}\varphi_q^{(0)}\otimes\left(\sum_{x_p\in\calQ_p/\calK_p}T_p(x_p)(j(g_\infty,i)^{-\kappa}f_0(g_\infty i)\varphi_0)\right)
 \end{align*}
 Taking the definition of $T_p$ into account, we have
 \[
 T_p(x_p)(j(g_\infty,i)^{-\kappa}f_0(g_\infty i)\varphi_0) = \tau_p(x_p)j(g_\infty,i)^{-\kappa}f_0(g_\infty i)\varphi_p^{(0)}
 \]
 and consequently the right-hand side of the equation above becomes
 \[
 \sum_{x_p\in\calQ_p/\calK_p}\tau_p(x_p)j(g_\infty,i)^{-\kappa}f_0(g_\infty i)\varphi_0
 \]
 such that
 \begin{align*}
   \langle \calT^{T_p}(F_f)(\gamma(g_\infty\times k)),\varphi_0\rangle &= \sum_{x_p\in\calQ_p/\calK_p}\tau_p(x_p)F_{f_0}(g\iota_p(x_p)) \\
   &=\int_{\calQ_p}\tau_p(x_p)F_{f_0}(g\iota_p(x_p))dx_p,
 \end{align*}
 where $F_{f_0}(\gamma(g_\infty \times k)) = j(g_\infty,i)^{-\kappa}f_0(g_\infty i)$ is the automorphic form associated to the scalar valued elliptic modular form $f_0\in S_\kappa(\Gamma)$ of weight $\kappa$ for $\Gamma$. The latter integral is the adelic Hecke operator attached to $t_p\in \calH(\calQ_p//\calK_p)$. It is important to notice that we only obtain those adelic Hecke operators where the underlying matrix $m(p^k,p^l)$ has a square as determinant. 
\end{remark}

\begin{lemma}\label{lem:properties_f}
  Let $p$ be a  prime,   $T_{k,l}\in \calH^+(\calQ_p//\calK_p,\omega_p)$ and $\calH^+(\calQ_p//\calK_p,\omega_p)$ as given in Corollary \ref{cor:struct_hecke_algebra} and Theorem \ref{thm:satake_isom_coprime}, respectively and $F\in A_{\kappa}(\omega_f)$.  Then $\calT^{T_{k,l}}$
  \begin{enumerate}
  \item[i)]
    is $\calG(\Q)$-invariant and 
  \item[ii)]
    fulfills
\begin{align*}
&    \calT^{T_{k,l}}(F)(gk) = \omega_f(k)^{-1}\calT^{T_{k,l}}(F)(g) \text{ for all } k\in \calK \text{  and all } g\in \calG(\A), \\
  & \calT^{T_{k,l}}(F)(zg) =\omega_f(z_f)^{-1}\calT^{T_{k,l}}(F)(g)  \text{ for all } z\in \calZ(\A) \text{  and all } g\in \calG(\A).
  \end{align*}
    \end{enumerate}
\end{lemma}

\begin{proof}
  \begin{enumerate}
  \item[i)]
  Since  $F\in A_{\kappa}(\omega_f)$  is $\calG(\Q)$-invariant,  the same holds for $\calT^{T_{k,l}}(F)$ as can be seen by means of \eqref{eq:hecke_op_adelich_components}.
  \item[ii)]
    By Theorem \ref{thm:correspondence_mod-adelic}, i), and Definition \ref{def:spherical_hecke_algebra}, ii) we have
    \begin{align*}
      \calT^{T_{k,l}}(F)(gk) &= \bigotimes_{q\not= p}F_q(g_qk_q)\otimes \left(\sum_{x_p\in \calQ_p/\calK_p}T_{k,l}(x_p)(F_p(g_pk_px_p))\right) \\
      &=\bigotimes_{q\not= p}\omega_q(k_q)^{-1}F_q(g_q)\otimes\left(\sum_{y_p\in \calQ_p/\calK_p}T_{k,l}(k_p^{-1}y_p)(F_p(g_py_p))\right)\\
      &=\begin{cases}
      \bigotimes_{q\not= p}\omega_q(k_q)^{-1}F_q(g_q)\otimes \omega_p(k_p)^{-1}\left(\sum_{y_p\in \calQ_p/\calK_p}T_{k,l}(y_p)(F_p(g_py_p))\right),& \text{ if } p| |D|\\
 \bigotimes_{q\not= p}\omega_q(k_q)^{-1}F_q(g_q)\otimes \left(\sum_{y_p\in \calQ_p/\calK_p}T_{k,l}(y_p)(F_p(g_py_p))\right),& \text{ if } p\nmid |D|    
      \end{cases} \\
      & = \omega_f(k)^{-1}\calT^{T_{k,l}}(F)(g).      
    \end{align*}
    For the second equation we used the substitution $y_p = k_px_p$. This settles the first claimed identity.
  
    For the second identity we make use of the fact that $F\in A_{\kappa}(\omega_f)$ and that $\omega_p(z_p)$ acts for $z_p\in \calZ(\Z_p)$ on $S_{L_p}$ by multiplication with a scalar (cf. \eqref{eq:weil_repr_scalar_matrix}), which commutes with the operator $T_{k,l}$. Let $z=z_\Q(z_\infty \times z_f)$ with $z_f = (z_q)_{q <\infty}\in \calK$. Then
    \begin{align*}
       \calT^{T_{k,l}}(F)(zg) &= \bigotimes_{q\not= p}\omega_q(z_q)^{-1}F_q(g_q)\otimes \left(\sum_{x_p\in \calQ_p/\calK_p}T_{k,l}(x_p)(\omega_p(z_p)^{-1}F_p(g_px_p))\right) \\
       &= \omega_f(z_f)^{-1}\calT^{T_{k,l}}(F)(g).
    \end{align*}
    \end{enumerate}
  \end{proof}

Let $p$ be a prime, $(k,l)\in \Lambda_+$, $T_{k,l}\in \calH^+(\calQ_p//\calK_p,\omega_p)$ and $\calH^+(\calQ_p//\calK_p,\omega_p)$ as in Corollary \ref{cor:struct_hecke_algebra} and Theorem \ref{thm:satake_isom_coprime}, respectively and  $\calT^{T_{k,l}}$    as in Definition \ref{deinition:adelic_hecke_op}. We now show that the map $\mathscr{A}$ commutes with the Hecke operators $\calT^{T_{k,l}}$ and $T(m(p^{-k},p^{-l}))$ on both sides and thereby confirm that $\calT^{T_{k,l}}$ indeed preserves $A_{\kappa}(\omega_f)$. 
For a prime $p\nmid |D|$ this result was in principle proved by Werner (cf. \cite{We}, Theorem 53), but not in our framework and not for a general Hecke operator $T(m(p^{-k},p^{-l}))$.

\begin{theorem}\label{thm;action_hecke_op_notcoprime}
  Let $p$ be a prime and $(k,l)\in \Lambda$.
  If $p$ divides $|D|$, let $T_{k, l}\in \calH^+(\calQ_p//\calK_p,\omega_p)$ as in Corollary \ref{cor:struct_hecke_algebra}. If $(p,|D|) = 1$, let $T_{k, l}=\mathbbm{1}_{\calK_pm(p^k, p^l)\calK_p}\cdot\id_{S_{L_p}}\in \calH(\calQ_p//\calK_p,\omega_p)$ be as in  Theorem \ref{thm:satake_isom_coprime}. Further, let  $\calT^{T_{k,l}}$  be  as in Definition \ref{deinition:adelic_hecke_op} and \newline$T(m(p^{-k},p^{-l}))$ the Hecke operator as defined in Section \ref{sec:preliminaries}. 
  Then for any $f\in S_\kappa(\rho_L)$ we have 
  \begin{equation}\label{eq;compatibility_notcoprime}
    \calT^{T_{k,l}}(F_f) = F_{p^{(k+l)(\frac{\kappa}{2}-1)}T(m(p^{-k},p^{-l}))f},
  \end{equation}
  where $F_f$ is the automorphic form related to $f$ via the map $\mathscr{A}$. 
\end{theorem}
\begin{proof}
  We know from Lemma \ref{lem:properties_f} that for any $g=\gamma(g_\infty\times k)\in \calG(\A)$ we have 
\begin{align*}
  \calT^{T_{k,l}}(F_f) (\gamma(g_\infty\times k)) &=  \calT^{T_{k,l}}(F_f)((g_\infty\times 1_f)(1\times k))\\
  &=\omega_f(k)^{-1} \calT^{T_{k,l}}(F_f)(g_\infty\times 1_f).
  \end{align*}
    The same holds for $F_{p^{(k+l)(\frac{\kappa}{2}-1)}T(m(p^{-k},p^{-l}))f}$ since it is an element of $A_{\kappa}(\omega_f)$. Hence, it suffices to prove \eqref{eq;compatibility_notcoprime} for $g= g_\infty\times 1_f$.
    
  The proof is an adaptation of the one of Lemma 3.7 in \cite{Ge}.  We have
  \begin{equation}\label{eq:adel_hecke_op}
    \begin{split}
      \calT^{T_{k,l}}(F_f)(g)&=\sum_{x_p\in \calQ_p/\calK_p}R^{T_{k,l}}(\iota_p(x_p))F_f(g\iota_p(x_p))\\
      &=\sum_{x_p\in \calK_pm(p^k,p^l)\calK_p/\calK_p}R^{T_{k,l}}(\iota_p(x_p))F_f(g\iota_p(x_p)),
    \end{split}
  \end{equation}
 where the last equation is due to Remark \ref{rem:prop_adelic_hecke_op}, i), and the fact that $T_{k,l}$ is supported on the double coset $\calK_pm(p^k,p^l)\calK_p$.  Following an idea of Gelbart, we set for $x_p\in \calK_pm(p^k,p^l)\calK_p/\calK_p$ 
  \begin{align*}
      & \gamma = (x_p,\dots, x_p,\dots)\in \calG(\Q),\\
      & k(x_p) = (x_p^{-1},\dots,x_p^{-1}, 1_p, x_p^{-1},\dots)\in \calK,\\
      & x_p^{-1}\in \calQ_\infty,
    \end{align*}
    where the $1_p$ in $k(x_p)$ is at $p$-th place.
    With these notations it is easily verified that
    \[
    \iota_p(x_p) = \gamma(x_p^{-1}\times k(x_p)).
    \]
    Therefore, the right-hand side of \eqref{eq:adel_hecke_op} becomes
     \begin{align*}
      \sum_{x_p\in \calK_pm(p^k,p^l)\calK_p/\calK_p}R^{T_{k,l}}(\iota_p(x_p))F_f(\gamma(x_p^{-1}g_\infty \times k(x_p))).
      \end{align*}
    Using the fact that  $F_f\in A_{\kappa}(\omega_f)$ and equation \eqref{eq:component_automorphic_form} subsequently, we find that the latter expression is equal to

    \begin{align*}
      &  \sum_{x_p\in \calK_pm(p^k,p^l)\calK_p/\calK_p}R^{T_{k,l}}(\iota_p(x_p))\omega_f(k(x_p))^{-1}F_f(x_p^{-1}g_\infty\times 1_f) \\
      &= \sum_{x_p\in \calK_pm(p^k,p^l)\calK_p/\calK_p}j(x_p^{-1}g_\infty,i)^{-\kappa}\sum_{\lambda\in D}f_\lambda(x_p^{-1}g_\infty i)R^{T_{k,l}}(\iota_p(x_p))(\omega_f(k(x_p))^{-1}\varphi_\lambda).
    \end{align*}
    Decomposing $R^{T_{k,l}}$ and $\omega_f$ into its local factors,  yields
    \begin{equation}\label{eq:adelic_hecke_op_local_comp}
      \begin{split}
      &      \sum_{x_p\in \calK_pm(p^{k},p^l)\calK_p/\calK_p}j(x_p^{-1}g_\infty,i)^{-\kappa}\times\\
      &\sum_{\lambda\in D}f_\lambda(x_p^{-1}g_\infty i)\bigotimes_{q\not=p}\omega_q(x_p^{-1})^{-1}\varphi_q^{(\lambda_q)}\otimes T_{k,l}(x_p)(\omega_p(1_p)^{-1}\varphi_p^{(\lambda_p)}).  \\
        \end{split}
    \end{equation}

To further simplify the right-hand side of \eqref{eq:adelic_hecke_op_local_comp}, we evaluate $\omega_q(x_p^{-1})^{-1}$ and $T_{k,l}(x_p)$ on a concrete set of representatives  $x_p$.  To this end, we first assume $k < l$. 
    It is easily seen that Lemma 13.4 of \cite{KL} carries over to our situation.
    Keeping this in mind, we can conclude that
    \begin{equation}\label{eq:right_cosets}
      \begin{split}
        &\left\{x_{s,b}=m(p^k,p^k)\kzxz{p^{s}}{b}{0}{p^{l-k-s}}\; |\; s=1,\dots l-k-1, b\in (\Z/p^s\Z)^\times\right\}\\
        &\cup \left\{x_b = m(p^k,p^k)\kzxz{p^{l-k}}{b}{0}{1}\;|\; b\in \Z/p^{l-k}\Z\right\}  \cup \left\{m(p^k,p^k)m(1,p^{l-k})\right\}
        \end{split}
      \end{equation}
is a set of representatives of $\calK_pm(p^{k},p^l)\calK_p/\calK_p$ for any prime $p$.  We now distinguish the cases $p\mid |D|$ and $p\nmid |D|$. The latter is easier and will be postponed to the end of the proof.

The decomposition 
\begin{equation}\label{eq:decomp_x_sb}
x_{s,b}^{-1} = \kzxz{r}{-b}{t}{p^s}m(p^{-k},p^{-l})n\_(-p^{l-k-s}t) \in \Gamma m(p^{-k},p^{-l})\Gamma
\end{equation}
with $rp^s+bt= 1$ and
\begin{equation}\label{eq:decomp_x_b}
x_b^{-1} = w m(p^{-k},p^{-l})w^{-1}n(-b)\in \Gamma m(p^{-k},p^{-l})\Gamma
\end{equation}
can easily be verified. Since $\Gamma\subset \calK_q$ for all primes $q$, this decomposition can also be interpreted as decomposition in $\calK_qm(p^{-k},p^{-l})\calK_q$
 for all primes $q$.  If $q\not=p$, we may utilize  Definition \ref{def:ext_local_weil_repr}  and obtain 
 \begin{equation}\label{eq:local_weil_rerpr}
   \begin{split}
  & \omega_q(x_{s,b}^{-1})^{-1} = \omega_q(n\_(-p^{l-k-s}t))^{-1}\omega_q(m(p^{-k},p^{-l}))^{-1}\omega_q(\kzxz{r}{-b}{t}{p^s})^{-1},\\
     & \omega_q(x_b^{-1})^{-1} = \omega_q(w^{-1}n(-b))^{-1}\omega_q(m(p^{-k},p^{-l}))^{-1}\omega_q(w)^{-1}.
     \end{split}
\end{equation}
 By Definition \ref{def:spherical_hecke_algebra}, we further have
 \begin{equation}\label{eq:local_hecke_algebra}
  \begin{split}
  & T_{k,l}(x_{s,b}) = \omega_p(n\_(-p^{l-k-s}t)^{-1})\circ T_{k,l}(m(p^{k},p^l))\circ \omega_p(\kzxz{r}{-b}{t}{p^s}^{-1}),\\
    & T_{k,l}(x_b) = \omega_p((w^{-1}n(-b))^{-1})\circ T_{k,l}(m(p^{k},p^l))\circ \omega_p(w^{-1}).
    \end{split}
 \end{equation}
 Following the proof of Theorem \ref{rem:hecke_algebra_mp}, i), we obtain
 \begin{equation}\label{eq:local_weil_m_p_k_l}
 \omega_q(m(p^{-k},p^{-l}))^{-1}\varphi_q^{(\mu_q)}  = \frac{g(D_q)}{g_{p^k}(D_q)}\varphi_q^{(p^{(l-k)/2}\mu_q)} = \frac{g(D_q)}{g_{p^l}(D_q)}\varphi_q^{(p^{(l-k)/2}\mu_q)},
 \end{equation}
 where for the last equation we have used that $p^{k+l}$ is a square and the last equation of \eqref{eq:weil_repr_scalar_matrix}.
 Moreover, comparing \eqref{eq:intertwiner_1} with \eqref{eq:local_weil_double_coset_1}, it becomes apparent that the identity
\[
T_{k,l}(m(p^k,p^l))\varphi_p^{(\mu_p)} = \omega_p(m(p^{-k},p^{-l}))^{-1}\varphi_p^{(\mu_p)}
\]
holds.
 Replacing $\omega_q^{-1}(x_p)$ and $T_{k,l}(x_p)$ in \eqref{eq:adelic_hecke_op_local_comp} with the expressions calculated before and piecing together the local Weil representations, we arrive at
 \begin{equation}\label{eq:finite_weil_m_p_k_l}
    \omega_f(m(p^{-k},p^{-l}))^{-1}\varphi^{(\lambda)} = \frac{g(D_p^\perp)}{g_{p^l}(D_p^\perp)}\varphi^{(p^{(l-k)/2}\lambda)}
\end{equation}
 and
\begin{equation}\label{eq:hecke_op_explicit}
  \calT^{T_{k,l}}(F)(g_\infty\times 1_f) = \sum_{x_p\in \calK_pm(p^{k},p^l)\calK_p/\calK_p}j(x_p^{-1}g_\infty,i)^{-\kappa}\sum_{\lambda\in D}f_\lambda(x_p^{-1}g_\infty i)\omega_f(x_p^{-1})^{-1}\varphi^{(\lambda)}.
   \end{equation}
On the other hand, it is well known that
\[
\{x_{s,b}^{-1}\;|\;s=1,\dots l-k-1, b\in (\Z/p^s\Z)^\times\}\cup \{x_b^{-1}\; |\;  b\in \Z/p^{l-k}\Z\} \cup \{m(p^k, p^l)^{-1}\}
\]
is a set of representatives of $\Gamma\bs\Gamma m(p^{-k},p^{-l})\Gamma$.
In view of \eqref{eq:decomp_x_sb} and \eqref{eq:decomp_x_b} we find

\begin{align*}
  & \rho_{L}(x_{s,b}^{-1})^{-1} =\rho_{L}(n\_(-p^{l-k-s}t))^{-1}\rho_{L}(m(p^{-k},p^{-l}))^{-1}\rho_{L}(\kzxz{r}{-b}{t}{p^s})^{-1}, \\
  & \rho_{L}(x_b^{-1})^{-1} = \rho_{L}(w^{-1}n(-b))^{-1} \rho_{L}(m(p^{-k},p^{-l}))^{-1}\rho_{L}(w)^{-1},
\end{align*}
where $\rho_{L}(m(p^{-k},p^{-l}))^{-1}\frake_\lambda = \frac{g(D_p^\perp)}{g_{p^l}(D_p^\perp)}\frake_{p^{(l-k)/2}\lambda}$ for all primes by \eqref{eq:weil_diagonal_matrix} or  \cite{St2}, (4.10). 

Thus, taking \eqref{eq:finite_weil_m_p_k_l} and  \eqref{eq:adelic_weil_rerpr} into account, we find that the right-hand side of \eqref{eq:hecke_op_explicit} equals
  \begin{align*}
   &  j(g_\infty,i)^{-k}\sum_{x\in \Gamma\bs\Gamma m(p^{-k},p^{-l})\Gamma} j(x,g_\infty i)^{-k}\sum_{\lambda\in D}f_\lambda(x(g_\infty i))\rho_L(x)^{-1}\frake_\lambda \\
    & =F_{(p^{2l})^{1-k/2}T(m(p^{-k},p^{-l}))f}(g_\infty\times 1_f). 
  \end{align*}

  For $k=l$ the set $\calK_p m(p^k,p^k)\calK_p/\calK_p$ consists  only of the element $m(p^k,p^k)$. Following the steps made before for the case $k<l$, one finds
  \begin{align*}
     \calT^{T_{k,k}}(F)(g_\infty\times 1_f) &= \frac{g(D_p^\perp)}{g_{p^k}(D_p^\perp)}j(g_\infty,i)^{-\kappa}\sum_{\lambda\in D}f_\lambda(g_\infty i)\varphi^{(\lambda)}\\
    &=\frac{g(D_p^\perp)}{g_{p^k}(D_p^\perp)}F_f(g_\infty\times 1_f).
    \end{align*}
  On the other hand, the Hecke operator $T(m(p^{-k},p^{-k}))$ acts just by multiplication with $\frac{g(D_p^\perp)}{g_{p^k}(D_p^\perp)}$, 
and once again the desired result follows. 

The proof for $p\nmid |D|$ starts again  with \eqref{eq:adelic_hecke_op_local_comp}. Let $T_{k,l} = \mathbbm{1}_{\calK_pm(p^k,p^k)\calK_p}\cdot\id_{S_{L_p}}$
Then, since $S_{L_p} = \C\varphi_p^{(0)}$ and $\omega_p$ is trivial, 
  \begin{align*}
    T_{k,l}(x_p)(\omega_p(1_p)^{-1}\varphi_p^{(0)}) &= T_{k,l}(m(p^k,p^l))\varphi_p^{(0)}\\
    &=\varphi_p^{(0)}.
  \end{align*}
  Therefore,
  \[
  \bigotimes_{q\not=p}\omega_q(x_p^{-1})^{-1}\varphi_q^{(\lambda_q)}\otimes T_{k,l}(x_p)(\omega_p(1_p)^{-1}\varphi_p^{(0)}) = \omega_f(x_p^{-1})^{-1}\varphi^{(\lambda)}.
  \]
  By means of \eqref{eq:local_weil_m_p_k_l} and \eqref{eq:finite_weil_m_p_k_l} and the decompositions of $x_p^{-1}$ above, the identity \eqref{eq;compatibility_notcoprime} follows as
\[
\rho_L(m(p^{-k},p^{-l}))^{-1}\frake_\lambda = \frac{g(D)}{g_p^l(D)}\frake_\lambda
\]
 (cf. \cite{St2}, (4.10) and note that $D(p^l) = D$ if $(p,|D|)=1$). 
  \end{proof}

\begin{remark}\label{lem:action_hecke_alg}
  \begin{enumerate}
    \item[i)]
The identity \eqref{eq;compatibility_notcoprime} can be rephrased with the help of the isomorphism $\mathscr{A}$. Let $F\in A_{\kappa}(\omega_f)$ with the associated modular form $f_F\in S_\kappa(\rho_L)$ and $f_{\calT^{T_{k,l}}(F)}$ the modular form corresponding to $\calT^{T_{k,l}}(F)$. Then \eqref{eq;compatibility_notcoprime} is equivalent to
\begin{equation}\label{eq:equiv_compatibility_rel_coprime}
f_{\calT^{T_{k,l}}(F)}  = p^{(k+l)(\kappa/2-1)}T(m(p^{-k},p^{-l}))(f_F). 
  \end{equation}
\item[ii)]
  It is also  an immediate but important consequence of  Theorem \ref{thm;action_hecke_op_notcoprime} 
  that $f\in S_\kappa(\rho_L)$ is a common eigenform for the Hecke operators $T(m(p^{-k},p^{-l}))$ for alle primes $p$ and all $(k,l)\in \Lambda$ if and only if the associated automorphic form $F_f$ is a common eigenform for the operators $\calT^{T_{k,l}}$ for all primes $p$ and all generators $T_{k,l}\in \calH^+(\calQ_p//\calK_p,\omega_p)$ ($\calH(\calQ_p//\calK_p,\omega_p)$ if $p$ and $|D|$ are coprime).   
  Remark \ref{rem:prop_adelic_hecke_op}, ii), allows us to extend this statement to the whole Hecke algebra $\calH^+(\calQ_p//\calK_p,\omega_p)$ (and $\calH(\calQ_p//\calK_p,\omega_p)$). Thus,
  \[
  \calT^{T}(F_f) = \lambda_{F_f,p}(T)F_f 
  \]
  for all $T\in \calH^+(\calQ_p//\calK_p,\omega_p)$ (and $\calH(\calQ_p//\calK_p,\omega_p)$) 
  if and only if  $f\in S_\kappa(\rho_L)$ is a common eigenform for all Hecke operators $T(m(p^{-k},p^{-l}))$. As in the classical scalar valued theory, we may then conclude that the map  
  \[
  \lambda_{F,p}:  \calH^+(\calQ_p//\calK_p,\omega_p)\rightarrow \C,\quad T\mapsto \lambda_{F,p}(T),
  \]
associated to an eigenform $F\in A_\kappa(\omega_f)$  defines a $\C$- algebra homomorphism. 
\end{enumerate}
\end{remark}

\end{document}